\documentclass{amsart}

\usepackage[margin=1.5in]{geometry}
\usepackage{amsmath}
\usepackage{amssymb}
\usepackage{amsthm}
\usepackage{enumerate}
\usepackage[colorlinks=true,
citecolor=blue,
linkcolor=red,
urlcolor=black,
anchorcolor=black]{hyperref}
\usepackage{xcolor}
\usepackage[english]{babel}
\usepackage{graphicx}
\usepackage{longtable}
\usepackage{tikz-cd}

\numberwithin{equation}{section}

\makeatletter
\newtheorem*{rep@theorem}{\rep@title}
\newcommand{\newreptheorem}[2]{%
\newenvironment{rep#1}[1]{%
 \def\rep@title{#2 \ref{##1}}%
 \begin{rep@theorem}}%
 {\end{rep@theorem}}}
\makeatother

\newtheorem{theorem}{Theorem}[section]
\newreptheorem{theorem}{Theorem}
\newtheorem{prop}[theorem]{Proposition}
\newtheorem{lemma}[theorem]{Lemma}
\newtheorem{cor}[theorem]{Corollary}
\newreptheorem{cor}{Corollary}
\newtheorem{defn}[theorem]{Definition}

\theoremstyle{definition}
\newtheorem{remark}[theorem]{Remark}

\newcommand{\Z}{\mathbb{Z}}
\newcommand{\F}{\mathbb{F}}
\newcommand{\Q}{\mathbb{Q}}
\newcommand{\C}{\mathbb{C}}
\newcommand{\R}{\mathbb{R}}

\newcommand{\Hyp}{\mathbb{H}}

\newcommand{\calB}{\mathcal{B}}
\newcommand{\calQ}{\mathcal{Q}}

\newcommand{\Qpsquare}{\Q_p^\times / (\Q_p^\times)^2}

\DeclareMathOperator{\Isom}{Isom}

\DeclareMathOperator{\Char}{char}
\DeclareMathOperator{\End}{End}

\title{Cusp types of arithmetic hyperbolic manifolds.}
\author{Duncan McCoy and Connor Sell}
\date{}

\begin{document}
\begin{abstract}
    We establish necessary and sufficient conditions for determining when a flat manifold can occur as a cusp cross-section within a given commensurability class of cusped arithmetic hyperbolic manifolds. This reduces the problem of identifying which commensurability classes of arithmetic hyperbolic manifolds can contain a specific flat manifold as a cusp cross-section to a question involving rational representations of the flat manifold's holonomy group. More generally we show that the holonomy representation provides an obstruction on the quasi-arithmetic manifolds containing a given flat manifold as a cusp cross-section. As applications, we prove that a flat manifold $M$ with a holonomy group of odd order appears as a cusp cross-section in every commensurability class of arithmetic hyperbolic manifolds if and only if $b_1(M)\geq 3$. We also provide examples of flat manifolds that arise as cusp cross-sections in a unique commensurability class of arithmetic hyperbolic manifolds and exhibit examples of pairs of flat manifolds that can never appear as cusp cross-sections in the same quasi-arithmetic hyperbolic manifold.
\end{abstract}
\maketitle

\section{Introduction}
\label{s:introduction}
The combined work of Long-Reid and McReynolds shows that every compact flat $n$-manifold arises as a cusp cross-section of a hyperbolic $(n+1)$-dimensional manifold \cite{LongReid, McReynolds2009covers}. In particular, for a given flat manifold $B$, they construct a commensurability class of arithmetic hyperbolic manifolds such that at least one manifold in this commensurability class contains a cusp diffeomorphic to $B\times \R_{\geq 0}$. In general, the manifolds containing $B$ as a cusp cross-section will have many other cusps. Indeed, there are examples of flat manifolds that can not arise as the cusp cross-section of a 1-cusped hyperbolic manifold \cite{Long-Reid2000geometric} and in sufficiently high dimensions all arithmetic manifolds necessarily have multiple cusps \cite{Stover2013ends}.

The principal result of this paper gives necessary and sufficient conditions for determining when a flat manifold can occur as a cusp cross-section within a given commensurability class of arithmetic hyperbolic manifolds. 
Any rational quadratic form $q$ of signature $(n+1,1)$ defines an associated commensurability class of arithmetic hyperbolic $(n+1)$-manifolds. Moreover all commensurability classes of cusped arithmetic hyperbolic manifolds arise this way when $ n>2 $ (see Section~\ref{s:arith_mflds}). Given any compact flat $n$-manifold $B$, there is an associated finite group $H$, the holonomy group of $B$, and an associated conjugacy class of representations $\rho: H \rightarrow GL_n(\Q)$, which we call the holonomy representations of $B$ (see Section~\ref{s:flat_mflds}).

\newcommand{\thmmaintechnical}{Let $B$ be a compact flat $n$-manifold and let $q$ be a rational quadratic form of signature $(n+1,1)$. Then $B$ arises as a cusp cross-section in the commensurability class of arithmetic hyperbolic manifolds defined by $q$ if and only if there is rational quadratic form $f$ such that
\begin{enumerate}[(i)]
\item $B$ has a holonomy representation with image in $O(f; \Q)$ and
\item the form $f\oplus \langle 1,-1\rangle$ is equivalent to $q$.
\end{enumerate}}
\begin{theorem}\label{thm:main_technical}
\thmmaintechnical
\end{theorem}

Suitably interpreted, the work of Long-Reid and McReynolds shows that if there is a holonomy representation of $B$ with image in $O(f;\Z)$, where $f$ is a positive-definite $n$-dimensional rational quadratic form, then $B$ appears as a cusp cross-section in the commensurability class of hyperbolic manifolds defined by the form $f\oplus \langle 1,-1 \rangle$. The main contribution of this paper is to show that if $B$ appears in the commensurability class of hyperbolic manifolds defined by a quadratic form $f\oplus \langle 1,-1 \rangle$, then $B$ has a holonomy representation taking values in the orthogonal group $O(f;\Q)$. In fact, our work applies to the broader class of quasi-arithmetic hyperbolic manifolds.
\newcommand{\quasiarithobstruction}{Let $ B $ be a compact flat $ n $-manifold and let $ f $ be a positive definite rational quadratic form. If $B$ appears as a cusp cross-section in a quasi-arithmetic manifold associated to the form $ f\oplus \langle 1, -1 \rangle $, then $B$ has a holonomy representation with image in $ O(f, \Q) $.}
\begin{theorem}\label{thm:quasi_arith_obstruction}
\quasiarithobstruction
\end{theorem}
Since properly quasi-arithmetic manifolds exist corresponding to any rational quadratic form $ q $ of rank $ \geq 4 $ \cite{Tho16}, Theorem~\ref{thm:quasi_arith_obstruction} yields obstructions to the occurrence of some flat $ n $-manifolds as a cusp cross-section in some non-arithmetic hyperbolic $ (n+1) $-manifolds for all $ n \geq 3 $.

From these results, we obtain obstructions to flat manifolds appearing as cusp cross-sections in commensurability classes of arithmetic hyperbolic manifolds in all dimensions by studying holonomy representations. Previously, such obstructions were known only for flat 3- and 4-manifolds. The second author classified the appearance of orientable flat 3-manifolds as cusp cross-sections in commensurability classes of arithmetic hyperbolic 4-manifolds and obtained obstructions for certain flat 4-manifolds to appear as cusp cross-sections \cite{Sell}.

For flat manifolds with holonomy group of odd order we can characterize precisely which ones appear as cusp cross-sections in every commensurability class of cusped arithmetic hyperbolic manifolds of the appropriate dimension.

\newcommand{\coroddholonomy}{Let $B$ be a flat $n$-manifold with holonomy group of odd order. Then $B$ appears as a cusp cross-section in every commensurability class of cusped hyperbolic arithmetic $(n+1)$-manifolds if and only if $b_1(B)\geq 3$. Moreover, if $b_1(B)\leq 2$, then there are infinitely many commensurability classes of hyperbolic manifolds that do not contain $B$ as a cusp cross-section.}
\begin{cor}\label{cor:odd_holonomy}
\coroddholonomy
\end{cor}

Since there are flat manifolds with holonomy of odd order and $b_1(B)\leq 2$ in all dimensions, we obtain the following.

\newcommand{\corobstructionalldims}{For all $n\geq 3$, there exist infinitely many commensurability classes of non-compact hyperbolic $(n+1)$-manifolds that do not contain every flat $n$-manifold as a cusp cross-section.}
\begin{cor}\label{cor:obstruction_all_dims}
\corobstructionalldims
\end{cor}
This complements the work of McReynolds who obtained obstructions to the appearance of cusp cross-sections for complex hyperbolic and quaternionic hyperbolic manifolds \cite{McReynolds2004peripheral}.

Moreover, it turns out that there are flat manifolds that occur as cusp cross-sections in precisely one commensurability class of arithmetic hyperbolic manifolds.

\newcommand{\coruniqueclass}{Let $B$ be a flat $n$-manifold with $b_1(B)=0$, $n\equiv 2 \bmod 4$ and holonomy group $H\cong (C_3)^k$ for some $k\geq 2$. Then $B$ appears as the as a cusp cross-section in a unique commensurability class of arithmetic hyperbolic $(n+1)$-manifolds. Moreover there are examples of such manifolds for all $n\geq 10$.}

\begin{cor}\label{cor:unique_class}
\coruniqueclass
\end{cor}

A more complete analysis of flat manifolds with holonomy $(C_3)^k$ and $b_1(B)=0$ is given in Theorem~\ref{thm:Z3b1=0}. We also study other flat manifolds with holonomy groups of the form $(C_p)^k$ for $p$ an odd prime. 

\newcommand{\thmprimeanalysis}{Let $p$ be an odd prime and let $B$ be a flat $n$-manifold with $b_1(B)=0$ and holonomy group $H\cong (C_p)^k$ for some $k\geq 2$. If $B$ appears as a cusp cross-section in a quasi-arithmetic hyperbolic manifold corresponding to a quadratic form $q$, then
    \[d(q)=\begin{cases}
        -1 &\text{if $n\equiv 0\bmod 2(p-1)$}\\
        -p &\text{if $n\equiv p-1\bmod 2(p-1)$.}
    \end{cases}\]}
\begin{theorem}\label{thm:prime_analysis}
    \thmprimeanalysis
\end{theorem}
Using manifolds of the form given by Theorem~\ref{thm:prime_analysis}, we are able to obtain restrictions on the topology of quasi-arithmetic hyperbolic manifolds by showing that there are pairs of flat manifolds that can never appear as cusp cross-sections in the same quasi-arithmetic hyperbolic manifold.  In particular, such cross-sections can never occur together on an arithmetic hyperbolic manifold.

\newcommand{\cortopobstruction}{For all even $n\geq 36$ there exist pairs of flat $n$-manifolds $B_1, B_2$ such that if $B_1$ and $B_2$ both appear as cusp cross-sections of a hyperbolic $(n+1)$-manifold $M$, then $M$ is not a quasi-arithmetic manifold.}

\begin{cor}\label{cor:top_obstruction}
    \cortopobstruction
\end{cor}

In \cite{Sell}, the second author used quaternion algebras to classify the occurrence of orientable flat 3-manifolds as cusp cross-sections of arithmetic hyperbolic 4-manifolds. One can easily recover this classification using Theorem~\ref{thm:main_technical} (see Theorem~\ref{thm:3-dim_analysis}). Moreover we are able to extend the analysis to classify the occurrence of the orientable flat 4-manifolds as cusp cross-sections in arithmetic commensurability classes. The obstructions to occurring in the commensurability class defined by a form $q$ are couched in terms of the discriminant $d(q)$ and the Hasse-Witt invariants $\epsilon_p(q)$. Here $\Q_p$ denotes the $p$-adic field. 
\newcommand{\thmfourdim}{Let $ B $ be a compact, orientable, flat 4-manifold with holonomy group $H$ and let $q$ be a rational quadratic form of signature $(5,1)$. Then $ B $ appears as a cusp cross-section in the commensurability class defined by $q$ if and only if one of the following conditions is satisfied:
    \begin{enumerate}[(i)] 
        \item $H$ is $C_1$, $C_2$ or $(C_2)^2$,
        \item $H$ is $C_4$ or $D_8$ and $ \epsilon_p(q) = 1 $ for all primes $p$ such that $ p \equiv 1\bmod 4$ and $-d(q)$ is a square in $\Q_p$,
        \item $H$ is $C_3$, $C_6$, $D_6$ or $D_{12}$ and $ \epsilon_p(q) = 1 $ for all primes $p$ such that $ p \equiv 1 \mod 3$ and $-d(q)$ is a square in $\Q_p$, or
        \item $H$ is $A_4$, and $ \epsilon_p(q) = 1 $ for all primes $p$ such that $-d(q)$ is a square in $\Q_p$.
    \end{enumerate}}

\begin{theorem}\label{thm:4-dim_analysis}
\thmfourdim
\end{theorem}

It is interesting to note that for flat 3- and 4-manifolds the obstructions to a given flat manifold $B$ appearing in a commensurability class depend only on the isomorphism type of the holonomy group $H$. This phenomenon does not persist into higher dimensions where the additional data of the holonomy representation is needed to determine if a flat manifold appears in a given commensurability class. For example, there are flat 5-manifolds $B$ and $B'$ both with holonomy group $C_3$ such that $b_1(B)=1$ and $b_1(B')=3$. Corollary~\ref{cor:odd_holonomy} shows that $B'$ appears as a cusp cross-section in every commensurability class of arithmetic manifolds, whereas $B$ does not.

\subsection*{Organization} In Section~\ref{s:quadratic_forms} we give a brief introduction to quadratic forms, and in Section~\ref{s:flat_mflds} we discuss properties of flat manifolds and the their holonomy representations. In Section~\ref{s:arith_mflds}, we give discuss arithmetic hyperbolic manifolds and prove Theorem~\ref{thm:main_technical}.  In Section~\ref{s:more_quadratic_forms}, we give further details on quadratic forms, and in Section~\ref{s:rep_theory}, we discuss some representation theory of finite groups. This further algebraic background allows us to prove the more tangible applications of Theorem~\ref{thm:main_technical} in Section~\ref{s:applications}.

\subsection*{Acknowledgements}
The first author would like to thank Steve Boyer and Michael Albanese for many helpful conversations.  The second author would like to thank Alan Reid for many helpful conversations and for introducing him to the problem.
DM is supported by an NSERC Discovery grant and a Canada Research Chair.

\section{Generalities on quadratic forms and orthogonal groups}
\label{s:quadratic_forms}
Let $\F$ be a field of characteristic zero and for a symmetric matrix $Q\in GL_n(\F)$ let $q$ be the non-degenerate quadratic form defined by 
\begin{align*}
    q:\F^n &\rightarrow \F\\
    v&\mapsto v^T Q v.
\end{align*}
The orthogonal group of $q$ is the matrix group
\[ 
O(q)=O(q;\F)=\{A\in GL_n(\F) \mid A^T Q A =Q\}.
\]
Given a subring $R\leq \F$ we set 
\[O(q;R)=O(q;\F)\cap GL_n(R).\]
Two quadratic forms $q_1$ and $q_2$ are equivalent over $ \F $ if there exists a matrix $C\in GL_n(\F)$ such that
\[
\text{$q_1(v)=q_2(Cv)$ for all $v\in \F^n$.}
\]
When $ q_1 $ and $ q_2 $ are equivalent over $ \mathbb{Q} $, we refer to them as \textit{rationally equivalent}.  Since we are assuming that $\Char \F=0$, every quadratic form over $\F$ is equivalent to a diagonal quadratic form. Throughout this article we use $\langle a_1, \dots, a_n \rangle$ to denote the diagonal quadratic form defined by the diagonal matrix with diagonal entries given by the $a_i$.

In general, equivalent forms have conjugate orthogonal groups.
\begin{prop}\label{prop:equivalence_implies_conjugation}
If $C\in GL_n(\F)$ is a matrix such that $q_1(v)=q_2(Cv)$ for all $v\in \F^n$, then 
\[
CO(q_1)C^{-1}=O(q_2).
\]
\end{prop}
\begin{proof}
    For $A\in O(q_1)$ and arbitrary $v\in \F^n$ we have that
    \[q_2(CAC^{-1}v)=q_1(AC^{-1}v)=q_1(C^{-1}v)=q_2(v).\]
    Thus $CAC^{-1}\in O(q_2)$ and we deduce that $CO(q_1)C^{-1}\leq O(q_2)$. The reverse inequality is easily obtained in a similar fashion.
\end{proof}
Given an $m\in \F$ and a quadratic form $q$, we use $mq$ to denote the form defined by rescaling by $m$. Note that if $m\neq 0$, then we have $ O(q) = O(mq) $. For forms defined over $\Q$, rescaling by $m\in \Q_{>0}$ leads to the notion of projective equivalence of forms.  
\begin{defn}[Projective equivalence]
Two quadratic forms $ f $ and $ g $ over $ \mathbb{Q} $ are \emph{projectively equivalent (over $ \mathbb{Q} $)} if there exists positive $ m \in \mathbb{Q}_{>0} $ such that $ mf $ and $ g $ are rationally equivalent.
\end{defn}
Projective equivalence will be further discussed in Section~\ref{s:proj_equiv}.

\section{Flat manifolds}
\label{s:flat_mflds}
Throughout this section we consider the semi-direct product $GL_n(\R)\ltimes \R^n$ as the group of  affine transformations of $\R^n$ in the obvious way. The subgroup $I\ltimes \R^n$, which we will simply refer to as $\R^n$, is the subgroup acting by translations. There is a split short exact sequence
\[
1\longrightarrow \R^n \longrightarrow GL_n(\R)\ltimes \R^n \overset{p}{\longrightarrow} GL_n(\R) \longrightarrow 1,
\]
where the map $p$ takes the affine transformation $x\mapsto Ax+ \mu$ to the matrix $A$.

Let $B$ a compact flat $n$-manifold with fundamental group $\Gamma$. Then $\Gamma$ is a Bieberbach group, that is, $\Gamma$ is a torsion-free group that embeds as a discrete subgroup of $O(n)\ltimes \R^n$. The third Bieberbach theorem implies that all such embeddings are conjugate in $GL_n(\R)\ltimes \R^n$. Given such an embedding the subgroup $T\leq \Gamma$ acting by translations on $\R^n$ is a normal abelian subgroup and the first Bieberbach theorem implies that $T$ is a maximal abelian subgroup of finite index and isomorphic to $\Z^n$. Thus we get a short exact sequence
\[
1\longrightarrow T \longrightarrow \Gamma \overset{p}{\longrightarrow} H \longrightarrow 1,
\]
where the finite group $H$ is the \emph{holonomy group} of $B$. Fixing an identification of $T$ with $\Z^n$, we obtain a faithful representation
\[
\rho_\mathrm{hol}: H\rightarrow GL_n(\Z),
\]
such that for $h\in H$, $\rho_\mathrm{hol}(h)$ is the matrix such that
\begin{equation}\label{eq:holonomy_equation}
\rho_\mathrm{hol}(h) t= gtg^{-1} \quad\text{for all $t\in T$ and all $g\in p^{-1}(h)$}.
\end{equation}
Since changing the identification of $T$ with $\Z^n$ changes the representation $\rho_\mathrm{hol}$ by conjugation in $GL_n(\Z)$, this construction defines a unique conjugacy class of representations for $H$.
\begin{defn}[Holonomy representation]
    Given a compact flat manifold $B$ with holonomy group $H$, a \emph{holonomy representation} for $B$ is any representation
    \[
    \rho: H\rightarrow GL_n(\R)
    \]
    conjugate to $\rho_\mathrm{hol}:H\rightarrow GL_n(\Z)$.
\end{defn}
Next we seek an alternative characterization of holonomy representations.
\begin{lemma}\label{lem:holonomy_properties}
    Let $\overline{\rho}: \Gamma \rightarrow O(f;\R)\ltimes \R^n$ be an injective homomorphism with discrete image, where $f$ is a positive-definite form of rank $n$ and $\Gamma$ is the fundamental group of a flat compact $n$-manifold $B$. Then there is a unique representation $\rho: H\rightarrow O(f;\R)$ such that the following diagram commutes
    \begin{equation}\label{eq:commuting_hol_square}
        \begin{tikzcd}
\Gamma \arrow{r}{\overline{\rho}}\arrow{d}{p} & O(f;\R)\ltimes \R^n \arrow{d}{p} \\
H \arrow{r}{\rho} & O(f;\R).
\end{tikzcd}
\end{equation}
Furthermore we have that
\begin{enumerate}[(i)]
    \item $\rho$ is a holonomy representation for $B$ and
    \item if $\overline{\rho}(T)\leq \Q^n$, then the image of $\rho$ is contained in $O(f;\Q)$.
\end{enumerate}
\end{lemma}
\begin{proof}
     Since $f$ is equivalent to the standard positive definite form of rank $n$ over $\R$, we may conjugate $\overline{\rho}$ by an element of $GL_n(\R)$ to obtain a homomorphism whose image is a Bieberbach group in $O(n)\ltimes \R^n$. In particular, this implies that $\overline{\rho}(T)$ forms a lattice in $I\ltimes \R^n$.
     
     Since $\overline{\rho}(T)$ is contained in the kernel of $p$, it follows that $\overline{\rho}$ descends to a unique $\rho : H\rightarrow O(f;\R)$ making \eqref{eq:commuting_hol_square}. Explicitly, for any $h\in H$ and any $g\in p^{-1}(h)$, we have that $\overline{\rho}(g)$ is the affine transformation of the form $x\mapsto \rho(h)x+ \mu$ for some $\mu\in \R^n$. On the other hand, for $t\in T$, we have that $\overline{\rho}(t)$ is a translation of the form $x \mapsto x+\lambda$ for some $\lambda\in \R^n$. A direct calculation shows that $\overline{\rho}(gtg^{-1})$ is the translation $x \mapsto x+ \rho(h)\mu$. Thus we see that $\rho(h)$ satisfies
    \begin{equation}\label{eq:induced_formula}
    \rho(h) \overline{\rho}(t)= \overline{\rho}(gtg^{-1}) \quad\text{for all $t\in T$ and all $g\in p^{-1}(h)$}.
    \end{equation}
    Since $\overline{\rho}(T)$ forms a lattice in $I\ltimes \R^n$, there exists a matrix $M\in GL_n(\R)$ such that $M\overline{\rho}(T)=\Z^n$. Thus we obtain an identification of $T$ with $\Z^n$ via the map $t \mapsto M\overline{\rho}(t)$.

    Using \eqref{eq:induced_formula} and computing $\rho_\mathrm{hol}$ as in \eqref{eq:holonomy_equation} using this identification shows that for any $h\in H$, we have that 
    \[
    \rho_\mathrm{hol}(h) M\rho(t)=M\rho(gtg^{-1}) = M\overline{\rho}(h)\rho(t) \quad \text{for all $t\in T$},
    \]
    where $g$ appearing in the intermediate step is an arbitrary element of $p^{-1}(h)$. It follows that $\rho_\mathrm{hol}=M\rho M^{-1}$. Thus $\rho$ is a holonomy representation for $B$.

    If $\overline{\rho}(T)$ is contained in $\Q^n$, then the matrix $M$ chosen above will be in $GL_n(\Q)$. So in this case $\rho$ is conjugate to $\rho_\mathrm{hol}$ by a matrix in $GL_n(\Q)$.
\end{proof}
In fact, we have a converse to Lemma~\ref{lem:holonomy_properties} in the sense that a holonomy representation can always be lifted to an embedding of $\Gamma$ in $GL_n(\R)\ltimes \R^n$
\begin{lemma}\label{lem:lifting_to_R}
    Let $\rho : H \rightarrow GL_n(\R)$ be a holonomy representation for a compact flat $n$-manifold with fundamental group $\Gamma$. Then there exists an injective homomorphism $\overline{\rho}$ with discrete image such that the following diagram commutes 
    \begin{equation*}
        \begin{tikzcd}
\Gamma \arrow{r}{\overline{\rho}}\arrow{d}{p} & GL_n(\R)\ltimes \R^n \arrow{d}{p} \\
H \arrow{r}{\rho} & GL_n(\R).
\end{tikzcd}
\end{equation*}
\end{lemma}
\begin{proof}
    Since $\Gamma$ is a crystallographic group, there exists injective $\overline{\sigma}:\Gamma \rightarrow O(n)\ltimes \R^n$ with discrete image. By Lemma~\ref{lem:holonomy_properties} this induces a holonomy representation $\sigma: H \rightarrow O(n)$ such that the diagram as in \eqref{eq:commuting_hol_square} commutes. By definition, any pair of holonomy representations are conjugate over $\R$ and there exists $M\in GL_n(\R)$ such that $\rho=M\sigma M^{-1}$. The homomorphism $\overline{\rho}=M\overline{\sigma}M^{-1}$ provides the necessary lift of $\rho$. 
\end{proof}

Crucially, this can be sharpened for integral holonomy representations. The argument in the following lemma is a key ingredient in the construction of Long and Reid \cite{LongReid}.
\begin{lemma}\label{lem:holonomy_lifting}
    Let $f$ be a positive definite quadratic form of rank $n$ and $B$ be a compact flat $n$-dimensional manifold with fundamental group $\Gamma$. Then $O(f;\Z)\ltimes \Z^n$ contains a subgroup isomorphic to $\Gamma$ if and only if 
    there is a holonomy representation $\rho: H\rightarrow GL_n(\R)$ with image contained in $O(f;\Z)$.
\end{lemma}
\begin{proof}
One implication is straightforward: given a subgroup of $O(f;\Z)\ltimes \Z^n$ isomorphic to $\Gamma$, Lemma~\ref{lem:holonomy_properties} constructs the desired holonomy representation.

We now set about finding $\Gamma$ as a subgroup of $O(f;\Z)\ltimes \Z^n$.
Let $\rho: H \rightarrow O(f;\Z)$ be a holonomy representation for $B$. Consider the set $\mathcal{S}$ of homomorphisms $\overline{\rho}$ such that the following diagram commutes:
    \begin{equation}\label{eq:int_hol_lifting}
        \begin{tikzcd}
\Gamma \arrow{r}{\overline{\rho}}\arrow{d}{p} & O(f;\Z)\ltimes \R^n \arrow{d}{p} \\
H \arrow{r}{\rho} & O(f;\Z).
\end{tikzcd}
\end{equation}
    Let $\langle g_1, \dots, g_\ell \mid r_1, \dots, r_m \rangle$ be a finite presentation for $\Gamma$. Under a homomorphism $\overline{\rho}$ such that \eqref{eq:int_hol_lifting} commutes the $\rho(g_i)$ are affine maps of the form $x\mapsto A_ix+\mu_i$ where $A_i=\rho(p(g_i))$ is determined and the $\mu_i\in \R^n$ satisfy the constraints imposed by the relations $r_1, \dots, r_m$. Since each relation $r_j$ is a product of the $g_i$, the constraints imposed on the $\mu_i$ are a system of linear equations coefficients arising as the products of the $A_i$. In particular, this allows us to identify $\mathcal{S}$ with an algebraic subset of $\R^{\ell n}$ defined by linear equations with coefficients in $\Z$. By Lemma~\ref{lem:lifting_to_R}, $\mathcal{S}$ is nonempty and contains $\overline{\rho}'$ which is injective and has discrete image. Thus the linear equations defining $\mathcal{S}$ also admit a solution over $\Q$ and we may, in fact, find rational solutions arbitrarily close to $\overline{\rho}'$. Thus there exists a homomorphism $\overline{\rho}:\Gamma \rightarrow O(f;\Z)\ltimes \Q^n$. Furthermore, since injectivity and discreteness impose open conditions on $\mathcal{S}$, we may take $\overline{\rho}$ to be injective and with discrete image.
    
    Finally, by rescaling the $\mu_i$ to clear denominators, we can obtain injective $\overline{\rho}$ with image in $O(f;\Z)\ltimes \Z^n$. The image of $\Gamma$ under this homomorphism yields the desired subgroup.
\end{proof}

Finally, we mention the toral extension construction of Auslander-Vasquez that allows us to construct new flat manifolds from old \cite{Vasquez1970flat}.
\begin{prop}\label{prop:toral_extension}
    Let $B$ be a compact flat manifold with holonomy representation $\rho: H\rightarrow GL_n(\Z)$. Then for any representation $\sigma: H\rightarrow GL_m(\Z)$, there is a compact flat $(n+m)$-manifold $\widetilde{B}$ with holonomy representation $\rho \oplus \sigma: H\rightarrow GL_{m+n}(\Z)$.
\end{prop}
\begin{proof}
We have an action of $\Gamma=\pi_1 (B)$ on $\R^n$ by isometries such that the quotient is $B$. The representation $\sigma$ gives an action of $H$ on $T^m=\R^m/\Z^m$. We may choose a Euclidean metric on $T^m$ so that $H$ acts by isometries. We have an action of $\Gamma$ on $\R^n\times T^m$ by Euclidean isometries via
\begin{align*}
    \R^n\times T^m &\rightarrow \R^n\times T^m\\
    (x,y)&\mapsto (gx,p(g)y)
\end{align*}
    for all $g\in \Gamma$, where $p:\Gamma \rightarrow H$ is the projection of $\Gamma$ onto its holonomy group. The quotient of $\R^n\times T^m$ by this action $\widetilde{B}$ is a compact flat manifold of dimension of $n+m$. Moreover, it is not hard to see that the holonomy group of $\widetilde{B}$ is $H$ and that the holonomy representation is $\rho\oplus \sigma: H\rightarrow GL_{m+n}(\Z)$.
\end{proof}

\section{Arithmetic hyperbolic manifolds}
\label{s:arith_mflds}
In order to maintain the convention that cusp cross-sections have dimension $n$, we will work with hyperbolic manifolds of dimension $n+1$.
Let $q_0$ denote the real quadratic form of signature $(n+1,1)$ defined by
\[
q_0=\langle \underbrace{1,\dots, 1}_{n+1},-1 \rangle.
\]
The hyperboloid model for $(n+1)$-dimensional hyperbolic space is the set
\[
\Hyp^{n+1}=\{ x\in \R^{n+2} \mid q_0(x)=-1, \, x_{n+2}>0\}
\]
equipped with the metric obtained by restricting the form $q_0$ to the tangent space of $\Hyp^{n+1}$. In this model of hyperbolic space the isometry group of $\Hyp^{n+1}$ is
\[\Isom(\Hyp^{n+1})= O^+(q_0;\R),\]
where $O^+(q_0;\R)$ is the index two subgroup of $O(q_0;\R)$ that maps $\Hyp^{n+1}$ to itself. 

The topic of this paper is the cross-sections $ B $ of the cusps of non-compact, finite-volume, arithmetic, hyperbolic $ (n+1) $-manifolds.  All non-compact, finite-volume, arithmetic, hyperbolic $ (n+1) $-manifolds are arithmetic of simplest type over $ \mathbb{Q} $ (\cite{DWM}, Cor. 18.6.3 gives the non-trialitarian case, and an explanation for the trialitarian case based on \cite{Tits65}, \cite{BeCl}, and \cite{Godement} can be found in \S 3.4 of \cite{BBKS}).  For this reason, rather than giving the definition of an arithmetic group in full generality, we give a definition for those of simplest type.  For similar reasons, we will also only define quasi-arithmetic manifolds analogous to arithmetic ones of simplest type.

Let $q$ be a rational quadratic form of signature $(n+1,1)$. Since real quadratic forms are classified by their signatures, there exists matrix $A\in GL_{n+2}(\R)$ such that $q_0(x)=q(Ax)$ for all $x\in \R^{n+2}$. Fixing such an $A$ allows us to define a new model for hyperbolic space
\[
\Hyp_q^{n+1}:= A \Hyp^{n+1},
\]
where the isometry group 
\[\Isom(\Hyp_q^{n+1})=O^+(q;\R)=A O^+(q_0;\R)A^{-1}\]
is the index two subgroup of $O(q;\R)$ that preserves $\Hyp_q^{n+1}$. Given a discrete, torsion-free subgroup $\Gamma \leq O^+(q;\R)$, we obtain a hyperbolic manifold $\Hyp^{n+1}_q /\Gamma$ such that the natural diffeomorphism $\Hyp^{n+1}/A^{-1}\Gamma A \rightarrow \Hyp^{n+1}_q/\Gamma$ induced by $A$ is an isometry.

\begin{defn}
An \emph{arithmetic hyperbolic $(n+1)$-manifold of simplest type over $\Q$} is a manifold of the form $ \Hyp_q^{n+1} / \Gamma $, where $ q $ is a rational quadratic form of signature $(n+1,1) $ and $ \Gamma \leq O^+(q; \mathbb{R}) $ is commensurable to $O^+(q; \Z)$. 
\end{defn}
An arithmetic hyperbolic manifold always has finite volume \cite{Borel-Harish-Chandra}. Since Selberg's lemma implies that $O^+(q;\Z)$ always contains a finite-index torsion-free subgroup, every rational quadratic form of signature $(n+1,1)$ defines a non-empty commensurability class of arithmetic hyperbolic manifolds. Furthermore, the commensurability class of hyperbolic manifolds defined by $q$ depends only on the projective equivalence class of $q$.

\begin{prop}[{\cite{MA15}}]
    Let $ M_1 $ and $ M_2 $ be arithmetic hyperbolic orbifolds of simplest type with associated quadratic forms $ q_1 $ and $ q_2 $ respectively.  Then $ M_1 $ and $ M_2 $ are commensurable if and only if $ q_1 $ and $ q_2 $ are projectively equivalent. \qed
\end{prop}
For the purposes of this article, the following definition of quasi-arithmetic hyperbolic manifold will be sufficient. A more general definition can be found in \cite{Vinberg}.
\begin{defn}
Let $q$ be a rational quadratic form of signature $(n+1,1)$. A \emph{quasi-arithmetic hyperbolic manifold associated to $q$} is a finite-volume hyperbolic manifold of the form $\Hyp_q^{n+1} / \Gamma $ where $ \Gamma \leq O^+(q; \mathbb{R}) $ is discrete, torsion-free and commensurable to a subgroup of $O^+(q; \Q)$.
\end{defn}

\subsection{Parabolic subgroups}
In the hyperboloid model of hyperbolic space, points in the boundary $\partial \Hyp^{n+1}$ correspond to lines in the cone $\{x \in \R^{n+2} \mid q_0(x)=0\}$. Similarly, if we take $q$ to be a rational quadratic form of signature $(n+1,1)$, then the points in the boundary $\partial\Hyp_q^{n+1}$ correspond to the lines in $\{x \in \R^{n+2} \mid q(x)=0\}$. Let $q=f\oplus \langle 1,-1\rangle$ be rational quadratic form of signature $(n+1,1)$. The vector $u=(0,\dots,0,1,1)^T$ satisfies $q(u)=0$ and so defines a point in $\partial\Hyp_q^{n+1}$ that will also be denoted by $u$. Let $\Lambda_u$ denote the subgroup of $O^+(q;\R)$ consisting of parabolic elements fixing $u$. A direct calculation shows that we have an isomorphism
\begin{align}\begin{split}\label{eq:explicit_iso}
    \Phi &: O(f;\R)\ltimes \R^n \rightarrow \Lambda_u\\
    \Phi&(A,v) =
\begin{pmatrix}
    A & -v & v\\
    v^TMA & 1-\frac{f(v)}{2} & \frac{f(v)}{2}\\
    v^TMA & -\frac{f(v)}{2} & 1 + \frac{f(v)}{2}
\end{pmatrix},
\end{split}
\end{align}
where $M$ is the symmetric matrix such that $f(v)=v^TMv$.

We have two important observations.
\begin{lemma}\label{lem:translations_unipotent}
    Let $M\in O^+(q;\R)$ be a non-trivial parabolic isometry of $\Hyp_q^n$ fixing a point $w\in \partial \Hyp_q^n$ that acts by translation on the horospheres tangent to $w$. Then
    \begin{enumerate}[(i)]
        \item the image of $(M-I)^2$ is the line spanned by $w$ and 
        \item $M$ is unipotent.

    \end{enumerate}
\end{lemma}
\begin{proof}
    Since the parabolic subgroup of $O^+(q;\R)$ fixing a $w\in \Hyp_q^{n+1}$ is conjugate to the parabolic subgroup $\Lambda_u$ in $O^+(q;\R)$, it suffices to verify the properties for a translation in $\Lambda_u$. In this case $X$ takes the form $M=\Phi(I,v)$ for $\Phi$ as in \eqref{eq:explicit_iso} with $v\neq 0$. For such an $M$, we have that
    \[
    (M-I)^2=\begin{pmatrix}
        0& 0& 0\\
        0 & -f(v) & f(v)\\
        0 & -f(v) & f(v),
    \end{pmatrix}
    \]
    whose image is clearly the span of $u$ since $f(v)\neq 0$ and
    \[
    (M-I)^3=0,
    \]
    giving unipotency.
\end{proof}
The key property of unipotent matrices we need is the following.
\begin{lemma}
    \label{lem:unipotent_rational}
    Let $M$ be a real unipotent matrix. If $M^k$ has coefficients in $\Q$ for some $k\geq 1$, then $M$ has also coefficients in $\Q$.
\end{lemma}

\begin{proof}
    A unipotent matrix $M$ has a minimal polynomial of the form $(X-1)^\ell$ for some positive integer $\ell$. We will show that the matrix $M^k$ is unipotent and has the same minimal polynomial. Unipotency of $M^k$ comes from the observation that
    \begin{equation}\label{eq:unipotency}
    (M^k-I)^\ell =(M-I)^\ell (M^{k-1}+ \dots + I)^\ell=0.
    \end{equation}
    Furthermore, \eqref{eq:unipotency} shows that the minimal polynomial for $M^k$ must divide $(X-1)^\ell$ and hence takes the form $(X-1)^{\ell'}$ for some $\ell'\leq \ell$.  
    On the other hand, given such a minimal polynomial for $M^k$, we see that the minimal polynomial of $M$, namely $(X-1)^\ell$,  must divide $(X^k-1)^{\ell'}$. This is only possible if $\ell'\geq \ell$. Hence $\ell=\ell'$ and we have shown that $M$ and $M^k$ have the same minimal polynomials. 

    Since the minimal polynomial of $M$ has degree $\ell$ and with rational coefficients, we see that the $\Q$-vector space spanned by the matrices $I, M, M^2, \dots $ is $\ell$-dimensional. Therefore, there is linear dependency over $\Q$ amongst the $\ell+1$ matrices $I, M^k, \dots M^{k(\ell-1)}$ and $M$. We write this dependency in the form
    \[b M= a_0 I + a_1 M^k + \dots + a_{\ell -1} M^{k(\ell-1)}.\]
    Since the minimal polynomial of $M^k$ has degree $\ell$, we see that $b$ must be non-zero. Thus we may write $M$ as a rational combination of $I, M^k, \dots, M^{k(\ell-1)}$, which are all rational matrices.
\end{proof}

 We say that $w\in \partial\Hyp^{n+1}_q$ is a \emph{rational point} if the line defining $w$ contains a non-zero point with rational coordinates. 
\begin{lemma}\label{lem:rational_transitivity}
    The group $O^+(q;\Q)$ acts transitively on the rational points in $\partial \Hyp_q^{n+1}$.
\end{lemma}
\begin{proof}
    This can be done by explicitly writing down appropriate rational matrices. For example, for any $w\in \Q^{n}$ we see that the rational point on the boundary corresponding to $(0,\dots, 0, -1, 1)^T$ is carried to $(2w, f(w)-1, f(w)+1)^T$ by a matrix of the form
    \[
    \begin{pmatrix}
    I & -w & w\\
    w^TM & 1-\frac{f(w)}{2} & \frac{f(w)}{2}\\
    w^TM & -\frac{f(w)}{2} & 1 + \frac{f(w)}{2}.
\end{pmatrix}\in O^+(q;\Q).
    \]
    However, the only rational point on the boundary that does not correspond to a rational point of this form is the point $(0,\dots, 0, 1, 1)^T$. However there are clearly elements of $O^+(q;\Q)$ that exchange $(0,\dots, 0, 1, 1)^T$ and $(0,\dots, 0, -1, 1)^T$, so we see that $O^+(q;\Q)$ acts transitively on the rational points in $\partial \Hyp^n_q$. 
\end{proof}

    

\subsection{Proof of Theorem~\ref{thm:main_technical} and Theorem~\ref{thm:quasi_arith_obstruction}}
\label{s:main_thm}
First, we show that the existence of a cusp cross-section implies the existence of a holonomy representation.
\begin{reptheorem}{thm:quasi_arith_obstruction}
\quasiarithobstruction
\end{reptheorem}
\begin{proof}
If $B$ appears as a cusp cross-section in such a quasi-arithmetic manifold, then there is a group $\Gamma\leq O^+(q;\R)$ commensurable with a  $O^+(q;\Q)$ such that for some $ y \in \partial \Hyp_q^n $ we have $\Pi=\Gamma \cap \Lambda_y$, where $\Pi \cong \pi_1(B)$ and $\Lambda_y$ denotes the maximal parabolic subgroup of $O^+(q;\R)$ fixing $y$.

Let $T\leq \Pi$ be the translation subgroup. First, we verify that $T\leq \Lambda_y\cap O^+(q;\Q)$. By Lemma~\ref{lem:translations_unipotent}(ii), any $t\in T$ is a unipotent matrix. Furthermore, since $\Gamma$ is commensurable with a subgroup of $O^+(q;\Q)$, there exists an integer $k\geq 1$ such that $t^k\in O^+(q;\Q)$. Thus Lemma~\ref{lem:unipotent_rational} implies that $t\in O^+(q;\Q)$.

Since $T\leq \Lambda_y\cap O^+(q;\Q)$ contains a non-trivial translation, Lemma~\ref{lem:translations_unipotent}(i) implies that $y$ is a rational point of $\partial \Hyp_q^n$. By Lemma~\ref{lem:rational_transitivity} this allows us to conjugate $\Gamma$ by an element of $O^+(q;\Q)$ in order to assume that $\Pi\leq \Lambda_u$, where $\Lambda_u$ is the maximal parabolic subgroup fixing $u=(0,\dots, 0,1,1)^T$. Note that we still have $T\leq \Lambda_u\cap O^+(q;\Q)$.

Applying $\Phi^{-1}$ from \eqref{eq:explicit_iso}, we obtain an injective homomorphism $\Phi^{-1}:\Pi \rightarrow O(f;\R)\ltimes \R^{n}$ with discrete image, where $T$ has image in the rational translations $I\ltimes \Q^n$. Applying Lemma~\ref{lem:holonomy_properties} to this homomorphism induces a holonomy representation of $B$ with image in $O(f;\Q)$, as required. 
\end{proof}
We make note of the following more elaborate statement from which Theorem~\ref{thm:main_technical} is easily deduced.
\begin{theorem}\label{thm:main_technical2}
    Let $f$ be a positive-definite $n$-dimensional rational quadratic form and let $B$ be a compact flat $n$-manifold. Then the following are equivalent:
    \begin{enumerate}[(i)]
        \item\label{it:cusp_cross} $B$ arises as a cusp cross-section in the commensurability class of arithmetic hyperbolic manifolds defined by $f\oplus \langle 1,-1 \rangle$;
        \item\label{it:rational_rep} $B$ has a holonomy representation taking values in $O(f; \Q)$ and
        \item\label{it:integral_rep} $B$ has a holonomy representation taking values in $O(f'; \Z)$ for a quadratic form $f'$ which is rationally equivalent to $f$.
    \end{enumerate}
\end{theorem}

\begin{proof}
The implication \eqref{it:cusp_cross}$\Rightarrow$\eqref{it:rational_rep} follows immediately from Theorem~\ref{thm:quasi_arith_obstruction}. 

Next we show \eqref{it:rational_rep}$\Rightarrow$\eqref{it:integral_rep}. Suppose that $\rho$ is a holonomy representation for $B$ taking values in $O(f;\Q)$. Since a holonomy representation is conjugate to an integral representation, then there is a matrix $T\in GL(n, \Q)$ such that $T^{-1} \rho T$ takes values in $GL_n(\Z)$. Thus $T^{-1} \rho T$ is a holonomy representation for $B$ taking values in $GL_n(\Z)\cap T^{-1} O(f;\Q) T$. However, $O(f';\Z)= GL_n(\Z)\cap T^{-1} O(f;\Q) T$, where $f'$ is the form satisfying $f'(x)=f(Tx)$ for all $x$.

Finally, we prove \eqref{it:integral_rep}$\Rightarrow$\eqref{it:cusp_cross}. This is the construction of \cite{LongReid}, as interpreted in Lemma~\ref{lem:holonomy_lifting}, combined with the subgroup separability results of \cite{McReynolds2009covers}.  If $B$ has a holonomy representation taking values in $O(f'; \Z)$, then Lemma~\ref{lem:holonomy_lifting} combined with the isomorphism in \eqref{eq:explicit_iso} shows that there is a subgroup of $\Pi\leq O^+(f'\oplus \langle 1,-1\rangle;\Z) \cap \Lambda_u$, where $\Pi\cong \pi_1(B)$ and $\Lambda_u$ is the parabolic subgroup fixing the point $u=(0,\dots,0,1,1)^T$. McReynolds' subgroup separability results \cite{McReynolds2009covers} imply that there is a finite index torsion free subgroup of $\Gamma \leq O^+(f'\oplus \langle 1, -1\rangle;\Z)$ such that $\Gamma \cap \Lambda_u=\Pi$. Thus $B$ arises as cusp cross-section in the commensurability class defined by $f'\oplus \langle 1, -1\rangle$. Since $f'\oplus \langle 1, -1\rangle$ and $f\oplus \langle 1, -1\rangle$ are rationally equivalent quadratic forms, they define the same commensurability class of hyperbolic manifolds.
\end{proof}

\begin{reptheorem}{thm:main_technical}
\thmmaintechnical
\end{reptheorem}
\begin{proof}
 Firstly, Theorem~\ref{thm:main_technical2} implies that if $B$ has a holonomy representation taking values in $O(f;\Q)$ for some positive definite $f$, then $B$ appears as a cusp cross-section in the commensurability class of hyperbolic manifolds defined by $f\oplus \langle 1, -1\rangle$. Consequently $B$ appears  as a cusp cross-section in the commensurability class defined by $q$ for any $q$ equivalent to $f\oplus \langle 1, -1\rangle$.

On the other hand, if $B$ appears as a cusp cross-section in the commensurability class of hyperbolic manifolds defined by $q$, then the manifolds in this class are non-compact. This implies that the form $q$ is isotropic \cite{Borel-Harish-Chandra}. However, this implies that $q$ is equivalent to $f\oplus \langle 1,-1\rangle$ for some positive definite quadratic form $f$ of rank $n$ \cite[Ch.~IV, Prop.~3]{Serre_arithmetic}. Theorem~\ref{thm:main_technical2} applies to show that $B$ has a holonomy representation taking values in $O(f;\Q)$.
\end{proof}

\section{Rational quadratic forms}
\label{s:more_quadratic_forms}
We begin by recalling some facts concerning the theory of rational quadratic forms and Hilbert symbols. For a detailed account we refer the reader to \cite{Serre_arithmetic}. For each prime $p$ we have the Hilbert symbol $(\_,\_)_p: \Q_p^\times \times \Q_p^\times \rightarrow \{\pm 1\}$, which is defined to be
\begin{equation*}
	(a,b)_p = \begin{cases}    
	 1 & \text{if $z^2 = ax^2 + by^2$ has a nonzero solution in $\Q_p$ } \\ -1 & \text{otherwise.}\end{cases}
\end{equation*}
The Hilbert symbol is clearly symmetric and satisfies the following useful properties:
\begin{enumerate}
    \item $(a,bc)_p=(a, b)_p(a,c)_p$ for all $a,b,c\in\Q_p^\times$;
    \item $(a^2, b)_p=1$ for all $a,b\in\Q_p^\times$ and
    \item $(a,-a)_p=1$ for all $a\in \Q_p^\times$.
\end{enumerate}
If we fix $a,b\in \Q^\times$ such that at least one of $a$ or $b$ is positive, then the global product formula \cite[Ch.~III, Thm.~3]{Serre_arithmetic} implies that the set of primes such that that $(a,b)_p=-1$ is finite and of even cardinality. Conversely, we have the following result.
\begin{prop}\label{prop:Hilbert_symbol_positive}
    Let $h_p\in \{\pm 1\}$ be a sequence of numbers indexed by the primes such that $h_p=-1$ for a finite even set of primes. For any $d\in\Q^\times$ such that $h_p=1$ whenever $d$ is a square in $\Q_p$, there exists $c\in \Q_{>0}$ such that $(d,c)_p=h_p$ for all $p$.
\end{prop}
\begin{proof}
    By \cite[Ch.~III, Thm.~4]{Serre_arithmetic}, there exists a $c\in\Q^\times$ such that $(d,c)_p=h_p$ for all $p$. Since there are an even number of $p$ such that $h_p=-1$, we have that $(c,d)_\infty=1$. Thus, if $d<0$, then $c$ is automatically positive. If $d>0$ and $c<0$, then the observation that $(c,d)_p=(c,d)_p(-d,d)_p=(-cd,d)_p$ for all primes $p$ shows that $-cd>0$ satisfies the necessary conclusions. 
\end{proof}

\subsection{Hasse-Witt invariants}
A non-degenerate rational quadratic form $q$ is determined up to rational equivalence by its signature $(r,s)$, its discriminant $d(f)\in \Q^\times / (\Q^\times)^2$ and a Hasse-Witt invariant $\epsilon_p(q)\in \{\pm 1\}$ for each prime $p$. There is also the invariant $\epsilon_\infty(q)$ which is determined by the signature, being equal to $(-1)^{\frac12 s(s-1)}$.  For a diagonal quadratic form $q=\langle a_1, \dots, a_n \rangle$, the discriminant and Hasse-Witt invariants are
\begin{equation}\label{eq:form_sum}
    \text{$d(q)=\prod_{i=1}^n a_i$ and $\epsilon_p(q)=\prod_{1\leq i<j\leq n} (a_i,a_j)_p$,}
\end{equation}
respectively. From these expressions it follows that for a pair of rational quadratic forms $f$ and $g$, we have
\[
\text{$d(f\oplus g)=d(f)d(g)$ and $\epsilon_p(f\oplus g)=\epsilon_p(f)\epsilon_p(g)(d(f),d(g))_p$.}
\]
When restricted to forms of rank at least three \cite[Ch. IV, \S~3.3]{Serre_arithmetic} yields the following conditions for the existence and uniqueness of rational quadratic forms.
\begin{theorem}
\label{thm:quadratic_classification}
    Let $r,s$ be integer such that $r+s\geq 3$ and $d\in\Q^\times/(\Q^\times)^2$. Let $h_p=\pm 1$ be a sequence indexed by the primes. Then there exists quadratic form $ q $ of signature $(r,s)$,  discriminant $ d $, signature $ (r,s) $, and Hasse-Witt invariants $ \epsilon_p(q)=h_p $ if and only if the following conditions are satisfied:
    \begin{enumerate}[(i)]
	\item $ \epsilon_p = 1 $ for almost all $ p $ and $ \prod_{p} \epsilon_p = (-1)^{\frac12 s(s-1)} $ and
	\item $ (-1)^s d>0 $.
    \end{enumerate}
    Moreover if such a form $q$ exists, then it is unique up to equivalence.\qed
\end{theorem}

\subsection{Projective equivalence}
\label{s:proj_equiv}
We can derive complete invariants for quadratic forms up to projective equivalence, by studying the behavior of the signature, discriminant, and Hasse-Witt invariants under scaling by a positive rational number $ m $.

It is clear that the signature of a form $ q $ is invariant under scaling by a positive number and that the discriminant of a rank $n$ form satisfies $ d(mq) = m^n d(q) $.  Since the discriminant lies in $ \mathbb{Q}^\times / (\mathbb{Q}^\times)^2 $, it is invariant under projective equivalence if and only if the rank $ n $ of $ q $ is even; if $ n $ is odd, any projective equivalence class contains quadratic forms of all possible discriminants.  The Hasse-Witt invariants of a scaled form depend on the value of $ n $ (mod 4).
\begin{lemma}
\label{lem:scaling_forms}
    Let $q$ be a quadratic form over $\Q$ of rank $n$, and $ m $ a positive rational number.  Then for any $ p $:
    \[\epsilon_p(mq)=
    \begin{cases}
        \epsilon_p(q) &\text{if $n\equiv 1\bmod 4$}\\
        \epsilon_p(q)(m,-d(q))_p &\text{if $n\equiv 2\bmod 4$}\\
        \epsilon_p(q)(m,-1)_p &\text{if $n\equiv 3\bmod 4$}\\
        \epsilon_p(q)(m,d(q))_p &\text{if $n\equiv 0\bmod 4$}\\
    \end{cases}
    \]
\end{lemma}
\begin{proof}
Without loss of generality we may assume that $q=\langle a_1, \dots, a_n \rangle$ has been diagonalized. Using the definition of the Hasse-Witt invariants and multiplicativity of the Hilbert symbol yields:
    \begin{align*}
        \epsilon_p(mq) & = \prod_{i<j} (ma_i, ma_j)_p \\ 
        & = \prod_{i<j} (m,m)_p (m,a_i)_p (m,a_j)_p (a_i, a_j)_p \\ 
        & = (m,m)_p^{\frac{n(n-1)}{2}} \left( \prod_i (m, a_i)_p^{n-1} \right) \left( \prod_{i<j} (a_i, a_j)_p \right) \\ 
        & = \left( m, m \right)_p^{\frac{n(n-1)}{2}} \left( m, \prod_i a_i^{n-1} \right)_p \epsilon_p(q) \\ 
        & = \left( m, m^{\frac{n(n-1)}{2}} \right)_p \left( m, d(q)^{n-1} \right)_p \epsilon_p(q) 
    \end{align*}
    When $ n $ is odd, $ (m, d(q)^{n-1})_p = (m, 1)_p = 1 $ and when $ n $ is even, $ (m, d(q)^{n-1})_p = (m, d(q))_p $.  We can make a similar argument to simplify $ (m, m^{\frac{n(n-1)}{2}} )_p $, observing that $ \frac{n(n-1)}{2} $ is even when $ n \equiv 0 $ or $ 1 \bmod 4 $, and odd otherwise. The proposition then follows from the identity $ (m, m)_p = (m, -1)_p $.
\end{proof}

We can combine our observations about the discriminant with Lemma~\ref{lem:scaling_forms} to write down a complete set of invariants for quadratic forms over $ \mathbb{Q} $ up to projective equivalence.

\begin{prop}
\label{prop:projective_invariants}
    Let $f$ and $g$ be two quadratic forms of rank $n$ over $ \mathbb{Q} $ with the same signature.
    \begin{itemize}
        \item If $n$ is odd, then $f$ and $g$ are projectively equivalent if and only if 
        \[(d(f),(-1)^{\frac{n-1}{2}})_p\epsilon_p(f)=(d(g),(-1)^{\frac{n-1}{2}})_p\epsilon_p(g)\] for all primes $p$.
        \item If $n$ is even, then $f$ and $g$ are projectively equivalent if and only if $d(f)=d(g)=d$ and $\epsilon_p(f)=\epsilon_p(g)$ for all primes $p$ such that $d=(-1)^{\frac{n}{2}}\in \Qpsquare$
    \end{itemize}
\end{prop}
\begin{proof}
If $n$ is odd, then it follows from Lemma~\ref{lem:scaling_forms} that the quantity $(d(f),(-1)^{\frac{n-1}{2}})_p\epsilon_p(f)$ is an invariant of the projective equivalence class. Conversely, if $f$ and $g$ have the same signature, then $d(f)$ and $d(g)$ have the same sign. In particular, $m=d(f)d(g)$ is positive. For such an $m$ we have that $d(mf)=d(g)$. For the Hasse-Witt invariants, we see that
\begin{align*}
    \epsilon_p(mf)&=(d(f)d(g), (-1)^{\frac{n-1}{2}})_p \epsilon_p(f)\\
    &=(d(f), (-1)^{\frac{n-1}{2}}) (d(g), (-1)^{\frac{n-1}{2}})_p \epsilon_p(f),
\end{align*}
for all primes.
Thus $mf$ is equivalent to $g$ if $(d(f),(-1)^{\frac{n-1}{2}})_p\epsilon_p(f)=(d(g),(-1)^{\frac{n-1}{2}})_p\epsilon_p(g)$ for all primes.

If $n$ is even, then the invariance of the discriminant and the stated Hasse-Witt invariants follows from Lemma~\ref{lem:scaling_forms} and the preceding discussion.
Conversely suppose that $f$ and $g$ have the same signature, discriminant $d$ and that $\epsilon_p(f)=\epsilon_p(g)$ for all primes $p$ such that $(-1)^{\frac{n}{2}}d$ is a square in $\Q_p$.

If we take $h_p=\epsilon_p(f)\epsilon_p(g)$ for all primes $p$, then this is a sequence such that $h_p=1$, whenever $(-1)^{\frac{n}{2}}d$ is a square in $\Q_p$. It follows from Proposition~\ref{prop:Hilbert_symbol_positive} that there exists positive $m\in \Q_{>0}$ such that $(m,(-1)^{\frac{n}{2}}d)_p=\epsilon_p(f)\epsilon_p(g)$ for all $p$. For such an $m$, we see that $mf$ and $g$ have the same Hasse-Witt invariants and are consequently equivalent forms.
\end{proof}

\subsection{Combining forms}
The aim of this section is to show that given three odd-dimensional positive-definite rational quadratic forms $f_1,f_2,f_3$, any positive definite rational form of the appropriate rank is equivalent to $a_1f_1 \oplus a_2 f_2 \oplus a_3 f_3$ for some $a_1,a_2,a_3\in \Q_{>0}$. 
\begin{lemma}
\label{lem:generalized_serre}
    Fix $ d \in \Q_{>0} $ a positive rational, fix $x,y,z \in \Q^\times$ nonzero rationals, and fix $h_p\in \{\pm 1\}$ a sequence of numbers indexed by the primes such that $h_p=-1$ for a finite even set of primes. Then there exist $a, b, c \in \Q_{>0}$ such that $ abc = d \in \mathbb{Q}^\times / (\mathbb{Q}^\times)^2 $ and for each $ p $,
    \begin{equation*}
        h_p = \epsilon_p(\langle a, b, c \rangle) (a, x)_p (b, y)_p (c, z)_p.
    \end{equation*}
\end{lemma}
\begin{proof}
    For each prime $p$, let 
    \[h^\prime_p = \epsilon_p(\langle x, y, z \rangle) (d, xyz)_p h_p.\]
    Theorem~\ref{thm:quadratic_classification} implies that there exist rationals $a^\prime, b^\prime, c^\prime\in \Q^\times$ such that the form $\langle a^\prime, b^\prime, c^\prime \rangle$ has Hasse-Witt invariants $\epsilon_p=h^\prime_p$ for every $p$, discriminant $dxyz$ and the same signature as $\langle x, y, z \rangle$. Since $\langle a^\prime, b^\prime, c^\prime \rangle$ and $\langle x, y, z \rangle$ have the same signature, we can assume that the quantities $ a = a^\prime x $, $ b = b^\prime y $, and $ c = c^\prime z $ are all positive. By construction, we have that
    \begin{equation*}
        abc = (a^\prime x)(b^\prime y)(c^\prime z) = (xyz)(xyzd) = d \in \mathbb{Q}^\times / (\mathbb{Q}^\times)^2,
    \end{equation*}
    Furthermore, we calculate that for any $p$ we have
    \begin{align*}
        \epsilon_p(\langle a, b, c \rangle) (a,x)_p (b,y)_p (c,z)_p & = (a,b)_p (a,c)_p (b,c)_p (a,x)_p (b,y)_p (c,z)_p \\
        & = (ax,by)_p (x,y)_p (ax,cz)_p (x,z)_p (by,cz)_p (y,z)_p (abc,xyz)_p \\
        & = \epsilon_p(\langle ax, by, cz \rangle) \epsilon_p(\langle x, y, z \rangle) (d, xyz)_p\\
        &=\epsilon_p(\langle a', b', c' \rangle) \epsilon_p(\langle x, y, z \rangle) (d, xyz)_p\\
        &=h'_p \epsilon_p(\langle x, y, z \rangle) (d, xyz)_p=h_p,
    \end{align*}
   as required.
\end{proof}

\begin{prop}\label{prop:all_forms}
    Let $ f_i $ be a positive definite rational quadratic form of odd rank $ r_i $ for $ i = 1,2,3 $ and let $g$ be a rational quadratic form of signature $(r,s)$. Then for any quadratic form $q$ of signature $ (r_1 + r_2 + r_3 + r,s)$, there exist positive rationals $ a_i $ such that $ q^\prime = a_1 f_1 \oplus a_2 f_2 \oplus a_3 f_3 \oplus g $ is rationally equivalent to $ q $.
\end{prop}
\begin{proof}
Since the $f_i$ are odd-dimensional, we may rescale them by a positive rational in order to assume that $d(f_i)=1$. Let $\delta_i=(-1)^{\frac{r_i-1}{2}}$.
Note that $d(g)$ and $d(q)$ both have the same sign as $(-1)^s$, thus $d(g)d(q)$ is positive and $(-d(q),d(g))_\infty=1$. Additionally we have that $\epsilon_\infty(g)=\epsilon_\infty(q)$ and $\epsilon_\infty(f_1)=\epsilon_\infty(f_2)=\epsilon_\infty(f_3)=1$ for signature reasons.
Thus if we set
\[
h_p=\epsilon_p(g)\epsilon_p(q) \epsilon_p(f_1)\epsilon_p(f_2)\epsilon_p(f_3)(-d(q),d(g))_p,
\]
the global product formula for the Hilbert symbol implies that $h_p$ takes the value -1 at a finite even number of primes. Thus Lemma~\ref{lem:generalized_serre} applies to give $a_1, a_2, a_3$ positive rationals such that
\[
\epsilon_p(\langle a_1, a_2, a_3 \rangle)(a_1,\delta_1)_p (a_2,\delta_2)_p (a_3,\delta_3)_p=h_p
\]
for all primes $p$ and $a_1a_2a_3=d(g)d(q)$. Using this choice of $a_1,a_2,a_3$, we set
\[
q'=a_1 f_1 \oplus a_2 f_2 \oplus a_3 f_3 \oplus g.
\]
Since the $f_i$ are all positive definite, the signatures of $ q $ and $ q^\prime $ will both be $ (r_1 + r_2 + r_3 + r,s) $. The discriminant satisfies
    \[
    d(q')=a_1a_2a_3 d(g)=d(q).
    \]
It remains to verify that the Hasse-Witt invariants of $q'$ and $q$ coincide. Using \eqref{eq:form_sum} and Lemma~\ref{lem:scaling_forms} we calculate the Hasse-Witt invariants of $ q^\prime $.  
    We first calculate that
\begin{align*}
        \epsilon_p(a_1 f_1 \oplus a_2 f_2 \oplus a_3 f_3) &= \epsilon_p(a_1 f_1)\epsilon_p(a_2 f_2 \oplus a_3 f_3)(a_1, a_2a_3)_p\\
        & = \epsilon_p(a_1 f_1)\epsilon_p(a_2 f_2 ) \epsilon_p(a_3 f_3)(a_1, a_2a_3)_p(a_2,a_3)_p\\
        & =\epsilon_p(\langle a_1, a_2, a_3 \rangle) (a_1, \delta_1)_p (a_2, \delta_2)_p (a_3, \delta_3)_p \epsilon_p(f_1)\epsilon_p(f_2 ) \epsilon_p(f_3)\\
        &=h_p\epsilon_p(f_1)\epsilon_p(f_2 ) \epsilon_p(f_3)\\
        &=\epsilon_p(g)\epsilon_p(q)(-d(q),d(g))_p.
\end{align*}
This in turn yields
\begin{align*}
        \epsilon_p(q^\prime) & = \epsilon_p(a_1 f_1 \oplus a_2 f_2 \oplus a_3 f_3)\epsilon_p(g)(a_1a_2a_3,d(g))_p \\
        & = \epsilon_p(g)\epsilon_p(q)(-d(q),d(g))_p \epsilon_p(g)(a_1a_2a_3,d(g))_p \\
        & = \epsilon_p(q)(-d(q),d(g))_p (d(q)d(g),d(g))_p\\
        &=\epsilon_p(q).
    \end{align*}
    Thus $q$ and $q'$ are equivalent rational quadratic forms.    
\end{proof}

\section{Orthogonal representations}
\label{s:rep_theory}
Throughout this section, $\rho:G\rightarrow GL_n(\F)$ will be a finite-dimensional representation of a finite group $G$ over a field $\F$ of characteristic 0. Given such a $\rho$, we will use $\calQ(\rho)$ to denote the set of all quadratic forms $q$ over $\F$ such that $\rho(G)\leq O(q;\F)$. Clearly the quadratic forms in $\calQ(\rho)$ are precisely those defined by a symmetric matrix $Q$ satisfying
\[
\text{$\rho(g)^TQ\rho(g)=Q$ for all $g\in G$.}
\]
We observe also that if $\rho'$ and $\rho$ are conjugate representations, then $\calQ(\rho)$ and $\calQ(\rho')$ contain exactly the same equivalence classes of quadratic forms. It follows from Proposition~\ref{prop:equivalence_implies_conjugation} that if $A^{-1}\rho A= \rho'$ for $A\in GL_n(\F)$, then a quadratic form $q$ is in $\calQ(\rho)$ if and only if the quadratic form $q'$ defined by $q'(v)=q(Av)$ is in $\calQ(\rho')$.

In order to apply Theorem~\ref{thm:main_technical} we will be particularly interested in representations over $\Q$.

For a rational representation $\rho:G\rightarrow GL_n(\Q)$, we always have a positive definite quadratic form in $\calQ(\rho)$. For example, we can obtain such a quadratic form by averaging the standard inner product:
\[
q(v)=\frac{1}{|G|}\sum_{g\in G} (\rho(g)v)^T (\rho(g)v).
\]
For rational representations the  most important conclusion of this section is that one can easily study $\calQ(\rho)$ by decomposing $\rho$ into sub-representations.

\begin{prop}\label{prop:rational_rep_decomp}
    Let $\rho:G\rightarrow GL_n(\Q)$ be a representation with a decomposition into irreducible representations over $\Q$ of the form
    \[
    \rho \cong \rho_1 \oplus \dots \oplus \rho_k.
    \]
    Then any non-degenerate quadratic form $q\in \calQ(\rho)$ is equivalent to one of the form
    \[
    q\cong q_1 \oplus \dots \oplus q_k
    \]
    where $q_i\in \calQ(V_i)$ is non-degenerate for $i=1,\dots, k$.
\end{prop}

\subsection{Invariant bilinear forms}
In order to study $\calQ(\rho)$, it will be convenient to expand the discussion to include more general bilinear pairings.
Given a representation $\rho:G\rightarrow GL_n(\F)$, a \emph{$\rho$-invariant (bilinear) pairing} is a bilinear pairing $b:\F^n\times \F^n \rightarrow \F$ such that $b(\rho(g)v,\rho(g)w)=b(v,w)$ for all $v,w\in \F^n$ and all $g\in G$. Given the natural equivalence between symmetric bilinear forms and quadratic forms, the elements of $\calQ(\rho)$ correspond to symmetric $\rho$-invariant pairings.

Recall that given a representation $\rho$, there is a dual representation $\rho^*$ defined by 
\[\rho^*(g)=\rho(g^{-1})^T.
\]
The representation $\rho$ is self-dual if it isomorphic to its dual.
\begin{lemma}\label{lem:self_dual}
   Let $\rho:G\rightarrow GL_n(\F)$ be a representation.
   \begin{enumerate}[(a)]
       \item\label{it:self_dual} There is a non-degenerate invariant pairing on $\rho$ if and only if $\rho$ is self-dual.
       \item\label{it:non_degene} If $\rho$ is irreducible, then a $\rho$-invariant pairing is non-degenerate if and only if it is non-zero.
   \end{enumerate}
\end{lemma}
\begin{proof}
   A bilinear pairing $b:\F^n\times \F^n \rightarrow \F$ is $\rho$-invariant if and only if the matrix $B$ such that $b(v,w)=v^TBw$ satisfies
   \begin{equation}\label{eq:invariant_pairing1}
    \text{$\rho(g)^T B \rho(g)=B$ for all $g\in G$.}
\end{equation}
However, $b$ is non-degenerate if and only if $B$ is invertible and so the existence of a non-degenerate $\rho$-invariant pairing is equivalent to the existence of an invertible matrix $B$ such that $\rho^*=B\rho B^{-1}$. This yields \eqref{it:self_dual}.

Given a matrix $B$ satisfying \eqref{eq:invariant_pairing1}, we see that the null-space of $B$ is a $\rho$-invariant subspace of $\F^n$. If $B$ is non-zero and $\rho$ is irreducible, then this implies that the null-space of $B$ is trivial. This in turn implies that $B$ is invertible and that the corresponding form $b$ is non-degenerate.
\end{proof}
The argument in Lemma~\ref{lem:self_dual} can easily be adapted to yield the following lemma.
\begin{lemma}\label{lem:cross_terms}
    Let $\rho: G\rightarrow GL_n(\F)$ and $\sigma:G\rightarrow GL_m(\F)$ be two irreducible representations. Then there exists a non-zero bilinear pairing $b: \F^n \times \F^m \rightarrow \F$ such that
    \[
    \text{$b(\rho(g)v,\sigma(g)w)=b(v,w)$ for all $v\in \F^n$, $w\in \F^m$ and $g\in G$}
    \]
    if and only if $\sigma$ is isomorphic to $\rho^*$. \qed
\end{lemma}
Let $\rho: G \rightarrow GL_n(\F)$ be a representation. Let $\End(\rho)$ denote the endomorphism algebra of $\rho$. That is,
\[
\End(\rho)=\{A\in M_n(\F) \mid \text{$\rho(g)A=\rho(g)A$ for all $g\in G$}\}.
\]
In general, the $\rho$-invariant pairings are all related via the endomorphism algebra of $\rho$
\begin{prop}\label{prop:endo_algebra}
Let $\rho: G \rightarrow GL_n(\F)$ be a self-dual representation and let $B$ be a matrix representing a non-degenerate $\rho$-invariant pairing. Then a matrix $B'$ represents a $\rho$-invariant pairing if and only if there exists $C\in \End(\rho)$ such that $B'=BC$.
\end{prop}
\begin{proof}
    By hypothesis, the matrix $B$ is invertible. Thus one can consider $C=B^{-1}B'$ and verify that it forms an element of $\End(\rho)$. Conversely given $C\in \End(\rho)$, we have for all $g$.
    \[
    \rho(g)^TBC\rho(g)=\rho(g)^TB\rho(g)C=BC.
    \]
    Thus $BC$ also defines a $\rho$-invariant pairing.
\end{proof}

For irreducible representations, we have the following trichotomy which extends the usual division of irreducible representations over $\C$ into complex, real and quaternionic representations (see \cite[Section~3.5]{FultonHarris}).
\begin{prop}\label{prop:general_trichotomy}
    Let $\rho:G \rightarrow GL_n(\F)$ be an irreducible representation. Precisely one of the following holds:
    \begin{enumerate}[(a)]
        \item ($\rho$ is complex) $\rho$ is not self-dual and admits no non-zero invariant bilinear pairing.
        \item ($\rho$ is real) $\rho$ is self-dual and admits at least one symmetric non-zero invariant bilinear pairing.
        \item ($\rho$ is quaternionic) $\rho$ is self-dual and every non-zero invariant bilinear pairing is skew-symmetric.
    \end{enumerate}
\end{prop} 
\begin{proof}
    By Lemma~\ref{lem:self_dual}, $\rho$ admits a non-zero invariant bilinear pairing if and only if $\rho$ is self-dual. Given a self-dual representation $\rho$, suppose that it admits a non-zero $G$-invariant pairing $b$ that is not skew-symmetric (i.e $\rho$ is not quaternionic). Then the symmetrized pairing defined by
    \[
    \overline{b}(v,w)=\frac{b(v,w)+b(w,v)}{2}
    \]
    for all $v,w\in V$ is a non-zero symmetric invariant bilinear pairing for $\rho$, implying that $\rho$ is real. 
\end{proof}

\subsection{Real representations}
The following lemma generalizes the fact that any symmetric bilinear form can be diagonalized over a field of characteristic 0. 
\begin{lemma}\label{lem:cross_terms_real}
    Let $\rho:G\rightarrow GL_n(\F)$ be an irreducible real representation for $G$. Then every quadratic form $q\in \calQ(\rho^{\oplus m}$) is equivalent to
    \[
    q\cong q_1 \oplus \dots \oplus q_m
    \]
    with $q_i\in \calQ(\rho)$ for $i=1, \dots, m$.
\end{lemma}
\begin{proof}
    Let $\calB$ denote the set of matrices $M\in GL_n(\F)$  representing $\rho$-invariant pairings. That is the set of $M$ such that $\rho(g)^T M \rho(g)=M$ for all $g\in G$. By Lemma~\ref{lem:self_dual}\eqref{it:non_degene} every non-zero $M\in \calB$ is invertible.
    
    The quadratic forms in $\calQ(\rho^{\oplus m})$ are represented by block matrices of the form
    \begin{equation}\label{eq:Q_form1}
        Q=(M_{ij})_{1\leq i, j\leq m},
    \end{equation}
    where  $M_{ij}\in \calB$ and $M_{ij}=M_{ji}^T$ for all $i$ and $j$.

    Note that if $D=(C_{ij})_{1\leq i, j\leq m}$ is an invertible block matrix such that each $C_{ij}$ is contained in $\End(\rho)$, then the matrix $D^T Q D$ represents a quadratic form in $\calQ(\rho^{\oplus m})$. The aim is to diagonalize $Q$ using such basis changes.

    If $Q$ is the zero form, then the proposition is clearly true. Thus we assume that $Q$ is non-zero.
    
    Thus there is some $M_{ij}$ which is non-zero. If $M_{ii}=0$ for all $i$, then we have $M_{ij}\neq 0$ for some $i\neq j$. By permuting the blocks we can assume that $M=M_{12}\neq 0$.
    Since $\rho$ is assumed to be real, there is a non-zero symmetric matrix $S\in \calB$. Given Proposition~\ref{prop:endo_algebra}, there is also $C\in \End(\rho)$ such that $MC=S$. 
    Performing a basis change of the following type allows us to assume that $M_{11}=2S$ is non-zero:
    \begin{equation*}
    \begin{pmatrix}
        I & C^T \\
        0 & I 
    \end{pmatrix}
    \begin{pmatrix}
        0 & M \\
        M^T & 0 
    \end{pmatrix}
    \begin{pmatrix}
        I & 0 \\
        C & I 
    \end{pmatrix}
    =
    \begin{pmatrix}
        MC+C^TM^T &  M\\
         M^T & 0
    \end{pmatrix}=
    \begin{pmatrix}
        2S &  M\\
         M^T & 0
    \end{pmatrix}.
    \end{equation*}

Thus we may assume that $M_{ii}\neq 0$ for some $i$ and hence by permuting blocks we may assume that $M_{11}\neq 0$. By Proposition~\ref{prop:endo_algebra}, for each $j$ there is a matrix $C_j=M_{11}^{-1}M_{1j}\in \End(\rho)$ such that $M_{11}C_j=M_{1j}$. Consequently, taking $C=\begin{pmatrix}
     C_2 & \dots & C_m
 \end{pmatrix}$ we may perform the following basis change.
 \[
\begin{pmatrix}
    I & 0 \\
    -C^T & I
\end{pmatrix}
Q\begin{pmatrix}
    I & -C\\
    0 & I
\end{pmatrix}=
\begin{pmatrix}
    M_{11} & 0\\
    0 & Q'
\end{pmatrix},
 \]
 where $Q'$ is a matrix representing a quadratic form in $\calQ(\rho^{\oplus (m-1)})$. This allows us to establish the result by induction on $m$.
\end{proof}

\begin{proof}[Proof of Proposition~\ref{prop:rational_rep_decomp}]
Since every representation over $\Q$ is real, it follows from Lemma~\ref{lem:cross_terms} and Lemma~\ref{lem:cross_terms_real} that any form $q\in \calQ(\rho)$ is equivalent to a form
\[
q\cong q_1 \oplus \dots \oplus q_k,
\]
where $q_i\in\calQ(\rho_i)$ for each $i$. Clearly $q$ is non-degenerate if and only if each of the $q_i$ is non-degenerate. 
\end{proof}

\subsection{Non-real representations}
Now we wish to understand $\calQ(\rho)$ when $\rho$ decomposes into non-real representations.
We will use following well-known algebraic fact. Recall that an isotropic subspace for a quadratic form $q$ is a subspace $U$ such that $q(v)=0$ for all $v\in U$.
\begin{lemma}\label{lem:isotropic_subspace}
    Let $q$ be a non-degenerate quadratic form of rank $2n$ over a field of characteristic $0$. If $q$ has an $n$-dimensional isotropic subspace, then $q$ is equivalent to the form
    \[
    q\cong\langle 1,-1 \rangle^{\oplus n}.
    \]
\end{lemma}
\begin{proof}
    By extending a basis for the isotropic subspace to a basis for the entire space, we see that $q$ can be represented by a matrix of the form
    \[
    M=\begin{pmatrix}
        0 & A\\ A^T & B,
    \end{pmatrix} 
    \]
    where the blocks are all $n\times n$ matrices and $B$ is symmetric. The non-degeneracy of $q$ implies that $A$ is invertible. Using $M$, we can explicitly diagonalize $q$. First we have that
    \begin{equation*}
    \begin{pmatrix}
        I &-\frac{1}{2}(BA^{-1})^T\\ 0 & I
    \end{pmatrix}^T
    \begin{pmatrix}
        0 & A\\ A^T & B
    \end{pmatrix}
    \begin{pmatrix}
        I &-\frac{1}{2}(BA^{-1})^T\\ 0 & I
    \end{pmatrix}=
    \begin{pmatrix}
        0 & A\\ A^T & 0
    \end{pmatrix}.
\end{equation*}
    This in turn can be diagonalized
    \begin{equation*}
    \begin{pmatrix}
       \frac{1}{2} I &-\frac{1}{2}I\\ A^{-1} & A^{-1}
    \end{pmatrix}^T
    \begin{pmatrix}
        0 & A\\ A^T & 0
    \end{pmatrix}
    \begin{pmatrix}
        \frac{1}{2} I &-\frac{1}{2}I\\ A^{-1} & A^{-1}
    \end{pmatrix}=
    \begin{pmatrix}
        I & 0\\ 0 & -I
    \end{pmatrix},
\end{equation*}
putting $q$ in the desired form.
\end{proof}
In the presence of only complex representations the analysis is straight-forward.
\begin{lemma}\label{lem:cross_terms_complex}
    Let $\rho:G\rightarrow GL_n(\F)$ be an irreducible complex (i.e non-self dual) representation for $G$. Then any non-degenerate quadratic form $q\in \calQ(\rho^{\oplus m}\oplus (\rho^*)^{\oplus m})$
 is equivalent to
    \[
    q\cong \langle 1, -1 \rangle^{\oplus nm}.
    \]
\end{lemma}
\begin{proof}
    Since $\rho$ is not self-dual, Lemma~\ref{lem:cross_terms} implies $q$ must be zero on the subspace corresponding to $\rho^{\oplus m}$. Thus we have a half-dimensional isotropic subspace and so Lemma~\ref{lem:isotropic_subspace} applies.
\end{proof}

The analysis for quaternionic representations is more delicate, but mimics the proof that every skew-symmetric form admits a symplectic basis.
\begin{lemma}\label{lem:cross_terms_quat}
    Let $\rho:G\rightarrow GL_n(\F)$ be an irreducible quaternionic representation for $G$. Every non-degenerate quadratic form $q\in\calQ(\rho^{\oplus m})$ is equivalent to
    \[
    q\cong \langle 1, -1 \rangle^{\oplus \frac12 n m}.
    \]
    Moreover such a form exists if and only if $m$ is even.
\end{lemma}
\begin{proof}
Let $\calB$ denote the set of matrices $M\in GL_n(\F)$  representing $\rho$-invariant pairings. That is the set of $M$ such that $\rho(g)^T M \rho(g)=M$ for all $g\in G$. By Lemma~\ref{lem:self_dual}\eqref{it:non_degene} every non-zero $M\in \calB$ is invertible.
    
The quadratic forms in $\calQ(\rho^{\oplus m})$ are represented by block matrices of the form
\begin{equation}\label{eq:Q_form2}
    Q=(M_{ij})_{1\leq i, j\leq m},
\end{equation}
    where  $M_{ij}\in \calB$ and $M_{ij}=M_{ji}^T$ for all $i$ and $j$. Since $\rho$ is quaternionic, the $M_{ij}$ satisfy $M_{ij}^T=-M_{ij}$.

   Since $\rho$ is quaternionic, there are no non-zero symmetric matrices in $\calB$. Thus there can be no non-degenerate symmetric forms when $n=1$. When $n=2$, the form necessarily has a half-dimensional isotropic subspace, allowing us to apply Lemma~\ref{lem:isotropic_subspace}.

   Since $M_{11}=0$ and $Q$ is non-degenerate, we have $M_{1j}\neq 0$ for at least one $j>1$. By permuting blocks we may assume that $M_{12}\neq 0$.

   By Proposition~\ref{prop:endo_algebra}, for each $j>2$ there is a matrix $C_j=M_{12}^{-1}M_{1j}\in \End(\rho)$ and a matrix $D_j=M_{12}^{-1}M_{2j}\in \End(\rho)$ such that $M_{12}D_j=M_{2j}$. Set $C$ and $D$ to be the block matrices
   \[
   \text{
   $C=\begin{pmatrix}
       C_3 & \dots & C_m
   \end{pmatrix}$ and
   $D=\begin{pmatrix}
       D_3 & \dots & D_m
   \end{pmatrix}$,}
   \]
respectively. We can perform the following basis change 
\begin{equation}\label{eq:0basischange}
    \begin{pmatrix}
        I & 0& 0 \\
        0 & I & 0 \\
        D^T & -C^T & I
    \end{pmatrix}
    Q
    \begin{pmatrix}
        I & 0 & D \\
        0 & I & -C\\
        0 & 0 & I
    \end{pmatrix}
    =
    \begin{pmatrix}
        0 &  M_{12} & 0 \\
         M_{12} & 0 & 0 \\
         0 & 0 & Q'
    \end{pmatrix},
    \end{equation}
    where $Q'$ represents a form in $\calQ(\rho^{m-2})$. This allows us to prove the lemma by induction on $m$.
\end{proof}

When applied to rational representations we obtain the following key result.

\begin{prop}\label{prop:rational_complex_only}
Let $G$ be a finite group of exponent $e$ and let $\rho: G\rightarrow O(q;\Q)$ be a representation of $G$ for $q$ a non-degenerate rational quadratic form of rank $n$. 
If the decomposition of $\rho$ into irreducible representations over $\C$ contains only non-real representations, then for any $p\equiv 1 \bmod e$ we have that $(-1)^\frac{n}{2}d(q)$ is a square in $\Q_p$ and $\epsilon_p(q)=1$.
\end{prop}
\begin{proof}
    It is a theorem of Brauer that the cyclotomic field $\Q[\zeta_e]$ is a splitting field for $G$ \cite[Theorem~24]{Serre_representations}. This means that any representation of $G$ that is irreducible over $\Q[\zeta_e]$ remains irreducible over any field extension of $\Q[\zeta_e]$. In particular, if we decompose $\rho$ into irreducible representations over $\Q[\zeta_e]$, then these representations remain irreducible over $\C$ and so are all non-real by hypothesis. By Lemma~\ref{lem:cross_terms}, Lemma~\ref{lem:cross_terms_complex} and Lemma~\ref{lem:cross_terms_quat}, this means that the form $q$ is equivalent to $\langle 1,-1\rangle^{\oplus \frac{n}{2}}$ over $\Q[\zeta_e]$. For $p>2$, the $p$-adic field $\Q_p$ contains $(p-1)$-roots of unity \cite[Ch.~II, Prop.~7]{Serre_arithmetic}. Therefore, for $p\equiv 1 \bmod e$ the $p$-adic field $\Q_p$ contains a copy of $\Q[\zeta_e]$. Thus $q$ is equivalent $\langle 1,-1\rangle^{\oplus \frac{n}{2}}$ over such a $\Q_p$, allowing us to compute the Hasse-Witt invariants and discriminant of $q$ in this field.
\end{proof}

\section{Applications}
\label{s:applications}
We wish to determine the commensurability classes of arithmetic hyperbolic $(n+1)$-manifolds containing a given flat $n$-manifold $B$ as a cusp cross-section. In light of Theorem~\ref{thm:main_technical}, this can be achieved by choosing a holonomy representation $\rho: H \rightarrow GL_n(\Q)$ for $B$ and classifying the quadratic forms $f\oplus \langle 1,-1\rangle$, where $f\in \calQ(\rho)$ is positive-definite. Given Proposition~\ref{prop:rational_rep_decomp}, we can simplify our calculations by decomposing $\rho$ into irreducible components over $\Q$ and studying each of these representations individually.

For a quadratic form $q$ of signature $(n+1,1)$ with $n\geq 3$, Theorem~\ref{thm:quadratic_classification} shows that $d(q)<0$ and the set of primes such that $\epsilon_p(q)=-1$ is finite and of even cardinality. Conversely, given any $d\in \Q_{<0}$ and a finite set of primes $S$ of even size, there is a (unique up to equivalence) quadratic form $q$ of signature $(n+1,1)$, discriminant $d(q)=d$ and $\epsilon_p(q)=-1\Leftrightarrow p\in S$.

\subsection{Constructions}
Firstly, we show that certain flat manifolds appear as cusp cross-sections in all commensurability classes of arithmetic hyperbolic manifolds of the appropriate dimension.
\begin{theorem}
    \label{thm:three_odd_irreps}
    Let $B$ be a flat $n$-manifold such that the holonomy representation $\rho: H\rightarrow GL_n(\Q)$ contains at least three odd-dimensional rational irreducible sub-representations. Then $B$ arises as a cusp cross-section in every commensurability class of cusped arithmetic hyperbolic $(n+1)$-manifolds.
\end{theorem}
\begin{proof}
    Since $\rho$ has at least three odd-dimensional rational sub-representations, we take $\rho$ in the form
    \[
    \rho= \rho_1 \oplus \rho_2 \oplus \rho_3 \oplus \rho_4,
    \]
    where $\rho_1, \rho_2, \rho_3, \rho_4$ are rational representations with $\rho_1, \rho_2, \rho_3$ odd-dimensional. For each the $\rho_i$ there is a positive-definite non-degenerate quadratic form $g_i$ such that $\rho_i: H \rightarrow O(g_i;\Q)$. Let $q$ be an arbitrary form of signature $(n+1,1)$. By Proposition~\ref{prop:all_forms} there exist $a_1,a_2,a_3\in \Q_{>0}$ such that
    \[
    q \cong a_1 g_1 \oplus a_2 g_2 \oplus a_3 g_3 \oplus g_4 \oplus \langle 1,-1\rangle.
    \]
    However the holonomy representation $\rho$ has image in 
    \[
    O(f';\Q)=O(a_1 g_1 \oplus a_2 g_2 \oplus a_3 g_3 \oplus g_4;\Q).
    \] 
    So Theorem~\ref{thm:main_technical} shows that $B$ appears as a cusp cross-section in the commensurability class defined by $q$.
\end{proof}
 In particular, Theorem~\ref{thm:three_odd_irreps} implies that if the holonomy representation of a flat $n$-manifold $B$ contains least three one-dimensional components over $\Q$, then $B$ appears as a cusp cross section in every commensurability class of cusped arithmetic hyperbolic $(n+1)$-manifolds. This applies to flat manifolds with $b_1(B)\geq 3$ and to those with holonomy groups of the form $H\cong (C_2)^k$.
\begin{cor}\label{cor:b_1geq3}
    Let $B$ be a flat $n$-manifold with $b_1(B)\geq 3$. Then $B$ arises as a cusp cross-section in every commensurability class of cusped arithmetic hyperbolic $(n+1)$-manifolds.
\end{cor}
\begin{proof}
    For a flat manifold $B$ the Betti number $b_1(B)$ is equal to the number of trivial representations when the holonomy representation is decomposed into irreducible representations \cite[Corollary~1.3]{Hiller1986flat}. Since trivial representations are rational and odd-dimensional the result follows from Theorem~\ref{thm:three_odd_irreps}.
\end{proof}

\begin{cor}\label{cor:Z2holonomy}
   Let $B$ be a flat $n$-manifold of dimension $n\geq 3$ and holonomy group $H\cong (C_2)^k$ for some $k\geq 0$.  Then $B$ appears as a cusp cross section in every commensurability class of cusped arithmetic hyperbolic $(n+1)$-manifolds.
\end{cor}
\begin{proof}
    Every irreducible representation of $(C_2)^k$ is rational and one-dimensional, so again Theorem~\ref{thm:three_odd_irreps} applies.
\end{proof}
\subsection{Manifolds with odd holonomy}
On the other hand it is possible to use Theorem~\ref{thm:main_technical} to obstruct certain flat manifolds from occurring as a cusp cross-section in some commensurability classes. Notably we are able to obtain obstructions in all dimensions.

\begin{theorem}\label{thm:odd_obstruction}
    Let $B$ be a compact flat $n$-manifold with holonomy group of odd order and $b_1(B)\leq 2$ and let $q$ be a rational quadratic form of signature $(n+1,1)$ such that $B$ appears in the commensurability class of arithmetic hyperbolic manifolds defined by $q$. Let $p$ be a prime such that $p\equiv 1 \bmod 4e$.
    \begin{enumerate}[(i)]
        \item\label{it:odd_b1_0} If $b_1(B)=0$, then $d(q)$ is a square in $\Q_p$ and $\epsilon_p(q)=1$.
        \item\label{it:odd_b1_1} If $b_1(B)=1$, then $\epsilon_p(q)=1$.
        \item\label{it:odd_b1_2} If $b_1(B)=2$ and $d(q)$ is a square in $\Q_p$, then $\epsilon_p(q)=1$.
    \end{enumerate}
\end{theorem}
\begin{proof}
It is a standard exercise using the Frobenius-Schur indicator to show that over $\C$ any irreducible representation of a group of odd order is either trivial or complex \cite[\S3.5]{FultonHarris}. On the other hand, $b_1(B)$ is equal to the number of trivial sub-representations of its holonomy representation. Let $m=b_1(B)$. Thus we can assume that the holonomy representation $\rho: H \rightarrow GL_n(\Q)$ of $B$ takes the form $\rho=\underbrace{\mathbf{1}\oplus \dots \oplus \mathbf{1}}_{m}\oplus \sigma$, where the decomposition of $\sigma$ into irreducible representations over $\C$ comprises solely of complex representations and $\mathbf{1}$ denotes the trivial representation. By Proposition~\ref{prop:rational_rep_decomp}, every positive definite form $f$ in $\calQ(\rho)$ is equivalent to $\langle a_1, \dots, a_{m} \rangle \oplus g$ for some $g\in \calQ(\sigma)$ and $a_1,\dots, a_m\in \Q_{>0}$.

Thus Theorem~\ref{thm:main_technical} shows that if $B$ appears in the commensurability class defined by a form $q$, then $q$ is equivalent to
\[q'=\langle a_1, \dots, a_{m} \rangle \oplus g \oplus \langle 1, -1\rangle,\]
with $a_1,\dots, a_m$ and $g$ as above.

Now fix a prime $p\equiv 1 \bmod 4e$. Since $p\equiv 1 \bmod 4$, we have that $-1$ is a square in $\Q_p$. Since $p\equiv 1 \bmod e$, Proposition~\ref{prop:rational_complex_only} implies that $d(g)=1 \in \Qpsquare$ and $\epsilon_p(g)=1$. In particular this implies that $q$ is equivalent to $\langle a_1, \dots, a_{m}, 1, \dots, 1\rangle$ over $\Q_p$. This allows us to easily calculate $d(q)\in \Qpsquare$ and $\epsilon_p(q)$. 
\begin{enumerate}[(i)]
    \item If $b_1(B)=0$, then it is immediate that $d(q)$ is a square in $\Q_p$ and that $\epsilon_p(q)=1$.
    \item Likewise if $b_1(B)=1$, then we immediately have $\epsilon_p(q)=1$.
    \item If $b_1(B)=2$, then $d(q)=a_1a_2 \in\Qpsquare$ and 
    \[\epsilon_p(q)=(a_1,a_2)_p=(a_1,-a_1a_2)_p=(a_1,d(q))_p.\]
    Thus we have $\epsilon_p(q)=1$ if $d(q)$ is a square in $\Q_p$.
\end{enumerate}
\end{proof}
\begin{remark}
    The hypothesis that $p\equiv 1 \bmod{4}$ in Theorem~\ref{thm:odd_obstruction} exists merely to simplify the analysis. One can obtain similar, but more complicated conditions for all primes $p\equiv 1 \bmod e$. 
\end{remark}

\begin{repcor}{cor:odd_holonomy}
    \coroddholonomy
\end{repcor}
\begin{proof}
If $b_1(B)\geq 3$, then Corollary~\ref{cor:b_1geq3} shows that $B$ appears as a cusp cross-section in every commensurability class of cusped arithmetic hyperbolic $(n+1)$-manifolds. Conversely, let $q$ be a quadratic form of signature $(n+1,1)$ with $d(q)=-1$ and $\epsilon_p(q)=-1$ for some prime $p\equiv 1\bmod 4|H|$, where $H$ is the holonomy group of $B$. If $b_1(B)\leq 2$, then Theorem~\ref{thm:odd_obstruction} shows that $B$ does not arise as a cusp cross-section in the commensurability class defined by $q$. Since $d(q)=-1$, Proposition~\ref{prop:projective_invariants} implies that $\epsilon_p(q)$ at such a prime is an invariant of the projective equivalence class of $q$. Since there are infinitely many primes satisfying $p\equiv 1\bmod 4|H|$, we may construct infinitely many distinct commensurability classes which do not contain $B$ as a cusp cross-section.  \end{proof}

There exist flat manifolds of odd holonomy group and $b_1(B)\leq 2$ in all dimensions $n\geq 3$. In dimension 3, there is the $\frac13$-twist manifold which has holonomy group $C_3$ and $b_1(B)=1$. Using toral extension, as in Proposition~\ref{prop:toral_extension}, with $\sigma$ the irreducible two-dimensional rational representation of $C_3$, we can construct flat manifolds with $b_1(B)=1$ and holonomy $C_3$ in all odd dimensions. Taking the product with $S^1$ produces flat manifolds with holonomy $C_3$ and $b_1(B)=2$ in all even dimensions. Thus we have the following.

\begin{repcor}{cor:obstruction_all_dims}
    \corobstructionalldims \qed
\end{repcor}

\subsection{The flat 3- and 4-manifolds}
We now leverage Theorem~\ref{thm:main_technical} to obtain a complete classification of which flat 3- and 4-manifolds appear as cusp cross-sections in commensurability classes of arithmetic hyperbolic manifolds. To do this we need to understand the irreducible rational representations of various groups of small order.
\begin{lemma}\label{lem:dim2_reps}
    Let $\rho:H \rightarrow GL_2(\Q)$ be representation of a finite group.
    \begin{enumerate}
        \item If $\rho(H)$ contains an element of order 3, then every form in $\calQ(\rho)$ is equivalent to $\langle 3a,a \rangle$ for some $a\in \Q$.
        \item If $\rho(H)$ contains an element of order 4, then every form in $\calQ(\rho)$ is equivalent to $\langle a,a \rangle$ for some $a\in \Q$.
    \end{enumerate}
\end{lemma}
\begin{proof}
    Any matrix of order three in $GL_2(\Q)$ is conjugate to $M_3=\begin{pmatrix}
        0 & -1 \\
        1 & -1
    \end{pmatrix}$. Thus, we may assume that image of $\rho$ contains $M_3$. An element of $\calQ(\rho)$ is represented by a symmetric matrix satisfying $M_3^TQM_3=Q$. One can verify that such a matrix necessarily takes the form $Q=\begin{pmatrix}
        2a & -a\\
        -a&2a
    \end{pmatrix}$. A quadratic form represented by such a $Q$ is equivalent to the diagonal form $\langle 3a,a\rangle$.

    Any matrix of order four in $GL_2(\Q)$ is conjugate to $M_4=\begin{pmatrix}
        0 & -1 \\
        1 & 0
    \end{pmatrix}$. Every symmetric matrix $Q$ satisfying $M_3^TQM_3=Q$ takes the form $Q=\begin{pmatrix}
        a & 0\\
        0&a
    \end{pmatrix}$ and hence corresponds to a diagonal quadratic form $\langle a,a\rangle$.
\end{proof}
We first recover the classification of cusp cross-sections for hyperbolic arithmetic 4-manifolds as derived by the second author \cite{Sell}. The compact orientable flat 3-manifolds are listed in Table~\ref{table:3-manifolds}. Note that for five-dimensional quadratic forms the projective invariants are simply the Hasse-Witt invariants $\epsilon_p(q)$ (see Proposition~\ref{prop:projective_invariants}).
\begin{theorem}\label{thm:3-dim_analysis}
    Let $q$ be a rational quadratic form of signature $(4,1)$. The commensurability class of arithmetic hyperbolic manifolds defined by $q$ contains the following cusp cross-sections: 
    \begin{enumerate}
        \item the 3-torus, $\frac{1}{2}$-twist manifold and the Hantsche-Wendt manifold; 
        \item the $\frac{1}{3}$-twist and $\frac{1}{6}$-twist manifolds if and only if $\epsilon_p(q)=1$ for all primes $p\equiv 1 \bmod{3}$ and
        \item the $\frac{1}{4}$-twist manifold if and only if $\epsilon_p(q)=1$ for all primes $p\equiv 1 \bmod{4}$.
    \end{enumerate}
\end{theorem}
\begin{proof}
We consider the cases in turn.

The 3-torus, $\frac{1}{2}$-twist manifold and the Hantsche-Wendt manifold all have holonomy of the form $(C_2)^k$. By Corollary~\ref{cor:Z2holonomy}, they appear as cusp cross-sections in the commensurability class defined by $q$.

For the $\frac{1}{3}$-twist and $\frac{1}{6}$-twist manifolds, we have $b_1(B)=1$ and an element of order three in the image of the holonomy representation, thus the holonomy representation must decompose over $\Q$ as $\rho\cong \rho_1 \oplus \mathbf{1}$, where $\rho_2$ is 2-dimensional and contains an element of order three. Thus using Lemma~\ref{lem:dim2_reps}, we see that these manifolds appear in the commensurability class defined by $q$ if and only if $q$ is equivalent to $\langle 3a,a,b,1,-1\rangle$ for some $a,b\in\Q_{>0}$.
The Hasse-Witt invariants of such a $q$ can be calculated to be
    \[
    \epsilon_p(q)= (3,-1)_p (-3,ab)_p.
    \]
    Since $-3$ is a square in $\Q_p$ if and only if $p\equiv 1 \mod 3$, this implies that $\epsilon_p(q)=(3,-1)_p$ whenever $p\equiv 1 \bmod 3$. Since $(3,-1)_p=1$ unless $p=2$ or 3, this implies that $\epsilon_p(q)=1$ whenever $p\equiv 1 \bmod 3$.
Conversely, every form $q$ of signature $(4,1)$ satisfying $\epsilon_p(q)=1$ for all $p\equiv 1\mod 3$ is equivalent to $\langle 3a,a,b,1,-1\rangle$ for some $a,b\in\Q_{>0}$. Proposition~\ref{prop:Hilbert_symbol_positive} gives $c>0$ such that $(-3,c)_p=\epsilon_p(q)(3,-1)_p$ for all primes $ p $. Setting $b=-3d(q)$ and $a=-3d(q)c$ yields a form with equivalent to $q$.

For the $\frac{1}{4}$-twist manifold we see that the holonomy representation must decompose over $\Q$ as $\rho\cong \rho_1 \oplus \mathbf{1}$, where $\rho_2$ is 2-dimensional and contains an element of order four. Using Lemma~\ref{lem:dim2_reps} again, we see that these manifolds appear in the commensurability class defined by $q$ if and only if is equivalent to $\langle a,a,b,1,-1\rangle$ for some $a,b\in\Q_{>0}$
    The Hasse-Witt invariants of such a $q$ can be calculated to be
    \[
    \epsilon_p(q)= (-1,ab)_p.
    \]
    Since $-1$ is a square in $\Q_p$ if and only if $p\equiv 1 \mod 4$, this implies that $\epsilon_p(q)=1$ whenever $p\equiv 1 \bmod 4$. 
    
    As in the previous case, Proposition~\ref{prop:Hilbert_symbol_positive} can be applied to show that every form of signature $(4,1)$ with $\epsilon_p(q)=1$ for all $p\equiv 1\bmod 4$ is equivalent to $\langle a,a,b,1,-1\rangle$ for some $a,b\in \Q_{>0}$.
\end{proof}

\begin{table}
\centering
\begin{tabular}{|c|c|c|l|} \hline
		$ B $ & $ \text{Hol}(B) $ & $b_1(B)$& Conditions on $ q $ \\ \hline
		3-torus/$ O^3_1 $  & $ C_1 $& $3$ & None \\ \hline
		$\frac12$-twist/$ O^3_2 $  & $ C_2 $& $1$ &None \\ \hline
		$\frac13$-twist/$ O^3_3 $  & $ C_3 $& $1$ & $\epsilon_p(q)=1$ for $p\equiv 1 \bmod{3}$  \\ \hline
		 $\frac14$-twist/$ O^3_4 $ & $ C_4 $& $1$ & $\epsilon_p(q)=1$ for $p\equiv 1 \bmod{4}$ \\ \hline
		 $\frac16$-twist/$ O^3_5 $ & $ C_6 $& $1$ & $\epsilon_p(q)=1$ for $p\equiv 1 \bmod{3}$\\ \hline
		Hantzsche-Wendt/$ O^3_6 $  & $ C_2^2 $& $0$ & None \\ \hline
\end{tabular}
\caption{The orientable flat 3-manifolds}
\label{table:3-manifolds}
\end{table}

Next we analyze the compact, orientable flat 4-manifolds. There are 27 homeomorphism classes of compact, orientable flat 4-manifolds, described in \cite{LRT13}. These manifolds are listed in Table~\ref{table:4-manifolds} along with their holonomy groups and first Betti numbers. This data along with the nomenclature is taken from \cite[Table~3]{LRT13}. Using similar analysis as for the 3-manifolds above, we find that 15 of these manifolds are obstructed from occurring as cusp cross-sections in some commensurability classes, and the remaining 12 occur in all commensurability classes.

For quadratic forms of rank six, Proposition~\ref{prop:projective_invariants} shows that invariants of the projective equivalence class are the discriminant $d(q)$ and $\epsilon_p(q)$ for all primes $p$ such that $-d(q)$ is a square in $\Q_p$.

\begin{reptheorem}{thm:4-dim_analysis}
    \thmfourdim
\end{reptheorem}

\begin{proof}
If $H$ contains no 3- or 4-torsion, then $H$ is one of $C_1$, $C_2$ or $(C_2)^2$. In these cases Corollary~\ref{cor:Z2holonomy} shows that $B$ appears in every commensurability class. Thus we may assume that $H$ contains 3- or 4-torsion. For each of these manifolds we wish to understand the decomposition of the holonomy representation into irreducible representations over $\Q$. First note that in all cases $b_1(B)\geq 1$, so the holonomy representation always contains at least one rational subrepresentation of dimension one. If $H$ contains 3- or 4-torsion, then the fact that $\rho$ is faithful implies $\rho$ must contain an irreducible rational subrepresentation of dimension at least two. However, the groups $C_4$, $C_3$, $C_6$, $D_8$, $D_6$ and $D_{12}$ do not admit any irreducible  three dimensional representations over $\Q$, so for these holonomy groups the holonomy representation decomposes over $\Q$ as $\rho \cong \rho_1 \oplus \rho_2\oplus \rho_3$, where $\rho_1$ is two dimensional and $\rho_2, \rho_3$ are one dimensional. 
\begin{itemize}
    \item If $H$ is isomorphic to $C_4$ or $D_8$, the the representation $\rho_1$ must contain an element of order 4. This means that $B$ appears in the commensurability class of $q$ if and only if $q$ is equivalent to a quadratic form $\langle a,a, b,c,1,-1\rangle$ for some $a,b,c\in \Q_{>0}$. Such a $q$ satisfies $d(q)=-bc$ and one can calculate its Hasse-Witt invariants to be
\[
\epsilon_p(q)=(ab,-1)_p(-b,-d(q))_p
\]
Thus, we see that $\epsilon_p(q)=1$ whenever both $-d(q)$ and $-1$ are squares in $\Q_p$. However $-1$ is a square in $\Q_p$ if and only if $p\equiv 1 \bmod 4$.

Conversely, any form $q$ of signature $(5,1)$ such that $\epsilon_p(q)=1$ for all $p$ such that $-d(q)$ and $-1$ are both squares in $\Q_p$ is equivalent to one of the form $\langle a,a, b,c,1,-1\rangle$ for some $a,b,c\in \Q_{>0}$. Proposition~\ref{prop:Hilbert_symbol_positive} allows us to pick $\alpha\in \Q_{>0}$ such that $(\alpha,-1)_p=\epsilon_p(q)$ for all primes where $-d(q)$ is a square in $\Q_p$. Thus we see that $\epsilon_p(q)(\alpha,-1)_p=1$ for all primes where $-d(q)$ is a square in $\Q_p$. This allows us to apply Proposition~\ref{prop:Hilbert_symbol_positive} again to obtain $\beta\in \Q_{>0}$ such that $(\beta,-d(q))_p=\epsilon_p(q)(\alpha,-1)_p$ for all primes. Consequently we have $(\beta d(q),-d(q))_p=\epsilon_p(q)(\alpha,-1)_p$ for all primes. If we set $b=-\beta d(q)>0$, $a=\alpha b$ and $c=\beta$, then this yields a form equivalent to $q$.

\item If $H$ is isomorphic to $C_3$, $C_6$, $D_6$ or $D_{12}$, then $\rho_1$ must contain an element of order 3. This means that $B$ appears in the commensurability class of $q$ if and only if $q$ is equivalent to $\langle 3a,a, b,c,1,-1\rangle$, for some $a,b,c\in \Q_{>0}$. Such a $q$ satisfies $d(q)=-3bc$ and
\[
\epsilon_p(q)=(ab,-3)_p(-3b,-d(q))_p(3,-1)_p.
\]
Thus, we see that $\epsilon_p(q)=(3,-1)_p$ whenever both $-d(q)$ and $-3$ are squares in $\Q_p$. However $-3$ is a square in $\Q_p$ if and only if $p\equiv 1 \bmod 3$. Since $(3,3)_p=-1$ if and only if $p=2,3$, this reduces to the condition that $\epsilon_p(q)=1$.

As in the previous case a double application of Proposition~\ref{prop:Hilbert_symbol_positive} can be used to show that any form $q$ of signature $(5,1)$ such that $\epsilon_p(q)=1$ for all $p$ such that $-d(q)$ and $-3$ are both squares in $\Q_p$ is equivalent to $\langle 3a,a, b,c,1,-1\rangle$
for some $a,b,c\in \Q_{>0}$.
\end{itemize}
Thus it remains to treat the case that $H$ is the alternating group $A_4$. Since we know the holonomy representation is faithful, we see that $\rho$ must contain a copy of the unique irreducible 3-dimensional representation of $A_4$. Thus the holonomy representation decomposes as
$\rho\cong \rho_1 \oplus \rho_2$, where $\rho_1$ is 3-dimensional and $\rho_2$ is 1-dimensional. By direct calculation one can show that every form in $\calQ(\rho_1)$ is equivalent to $\langle a,a,a \rangle$ for some $a\in \Q$. This means that $B$ appears in the commensurability class of $q$ if and only if $q$ is equivalent to $\langle a,a, a,b,1,-1\rangle$ for some $a,b,c\in \Q_{>0}$. Such a $q$ satisfies $d(q)=-ab$ and
\[
\epsilon_p(q)=(-a,-d(q))_p.
\]
Thus we have that $\epsilon_p(q)=1$ whenever $-d(q)$ is a square in $\Q_p$. Once again, Proposition~\ref{prop:Hilbert_symbol_positive} can be used to obtain the converse. If $q$ is a quadratic form of signature $(5,1)$ such that $\epsilon_p(q)=1$ whenever $-d(q)$ is a square in $\Q_p$, then there is $b\in \Q_{>0}$ such that $\epsilon_p(q)=(b,-d(q))_p=(b d(q),-d(q))_p$ for all $p$. If we set $a=-b d(q)$, then $q$ is equivalent to $\langle a,a, a,b,1,-1\rangle$.
\end{proof}
\begin{remark}
We see that for each discriminant $d<0$, there is a unique commensurability class containing the manifolds with holonomy $A_4$ as cusp cross-sections.
\end{remark}

\begin{table}
\centering
\begin{tabular}{|c|c|c|c|l|} \hline
		$ B $ & $ \text{Hol}(B) $ & $b_1(B)$& $\rho$ irreducibles& Conditions on $ q $ \\ \hline
		$ O^4_1 $ & $ C_1 $& $4$ & $1+1+1+1$ & None \\ \hline
		$ O^4_2 $ & $ C_2 $& $2$ & $1+1+1+1$ &None \\ \hline
		$ O^4_3 $ & $ C_2 $& $2$ & $1+1+1+1$& None \\ \hline
		$ O^4_4 $ & $ C_3 $& $2$ & $2+1+1$ &$ \epsilon_p(q) = 1 $ for $ p \equiv 1 \bmod 3 $ \\ \hline
		$ O^4_5 $ & $ C_3 $& $2$ & $2+1+1$ &$ \epsilon_p(q) = 1 $ for $ p \equiv 1 \bmod 3 $ \\ \hline
		$ O^4_6 $ & $ C_4 $& $2$ & $2+1+1$ &$ \epsilon_p(q) = 1 $ for $ p \equiv 1 \bmod 4 $  \\ \hline
		$ O^4_7 $ & $ C_4 $& $2$ & $2+1+1$ &$ \epsilon_p(q) = 1 $ for $ p \equiv 1 \bmod 4$ \\ 
        \hline
		$ O^4_8 $ & $ C_6 $& $2$ & $2+1+1$ &$ \epsilon_p(q) = 1 $ for $ p \equiv 1 \bmod 3 $ \\ \hline
		$ O^4_9 $ & $ (C_2)^2 $& $1$ &$1+1+1+1$& None \\ \hline
		$ O^4_{10} $ & $ (C_2)^2 $& $1$ &$1+1+1+1$& None \\ \hline
		$ O^4_{11} $ & $ (C_2)^2 $&$1$ &$1+1+1+1$& None \\ \hline
		$ O^4_{12} $ & $ (C_2)^2 $& $1$ &$1+1+1+1$& None \\ \hline
		$ O^4_{13} $ & $ (C_2)^2 $& $1$ &$1+1+1+1$& None \\ \hline
		$ O^4_{14} $ & $ (C_2)^2 $& $1$ &$1+1+1+1$& None \\ \hline
		$ O^4_{15} $ & $ (C_2)^2 $& $1$ &$1+1+1+1$& None \\ \hline
		$ O^4_{16} $ & $ (C_2)^2 $& $1$ &$1+1+1+1$& None \\ \hline
		$ O^4_{17} $ & $ (C_2)^2 $& $1$ &$1+1+1+1$& None \\ \hline
		$ O^4_{18} $ & $ D_6 $ & $1$& $2+1+1$& $ \epsilon_p(q) = 1 $ for $ p \equiv 1 \bmod 3 $ \\ \hline
		$ O^4_{19} $ & $ D_6 $ &$1$&$2+1+1$& $ \epsilon_p(q) = 1 $ for $ p \equiv 1 \bmod 3 $ \\ \hline
		$ O^4_{20} $ & $ D_6 $ &$1$&$2+1+1$& $ \epsilon_p(q) = 1 $ for $ p \equiv 1 \bmod 3 $ \\ \hline
		$ O^4_{21} $ & $ D_8 $ &$1$&$2+1+1$& $ \epsilon_p(q) = 1 $ for $ p \equiv 1 \bmod 4 $ \\ \hline
		$ O^4_{22} $ & $ D_8 $ &$1$&$2+1+1$& $ \epsilon_p(q) = 1 $ for $ p \equiv 1 \bmod 4$ \\ \hline
		$ O^4_{23} $ & $ D_8 $ &$1$&$2+1+1$& $ \epsilon_p(q) = 1 $ for $ p \equiv 1 \bmod 4$ \\ \hline
		$ O^4_{24} $ & $ D_8 $ &$1$&$2+1+1$& $ \epsilon_p(q) = 1 $ for $ p \equiv 1 \bmod 4$ \\ \hline
		$ O^4_{25} $ & $ D_{12} $ &$1$&$2+1+1$& $ \epsilon_p(q) = 1 $ for $ p \equiv 1 \bmod 3 $\\ \hline
		$ O^4_{26} $ & $ A_4 $ &$1$& $3+1$ &$ \epsilon_p(q) = 1 $  \\ \hline
		$ O^4_{27} $ & $ A_4 $ &$1$& $3+1$& $ \epsilon_p(q) = 1 $  \\ \hline
\end{tabular}
\caption{The orientable flat 4-manifolds. The conditions on $q$ apply to the $\epsilon_p$ only at the primes such that $d(q)$ is a square in $\Q_p$. The column entitled ``$\rho$ irreducibles'' lists the dimensions of the irreducible components of the holonomy representation over $\Q$.}
\label{table:4-manifolds}
\end{table}
 
\subsection{Holonomy of the form $(C_3)^k$}
In cases where the holonomy group is sufficiently simple it is possible to fully determine the commensurability classes of arithmetic hyperbolic manifolds containing a given manifold. We concentrate on the case of $b_1(B)=0$, since this yields the most interesting examples, but similar analysis can carried out to obtain obstructions when $b_1(B)=1$ or $2$, as well.
\begin{theorem}\label{thm:Z3b1=0}
    Let $B$ be a flat $n$-manifold with holonomy group of the form $(C_3)^k$ and $b_1(B)=0$. Let $q$ be a rational quadratic form of signature $(n+1,1)$. Then $B$ appears as a cusp cross-section in the commensurability class of hyperbolic manifolds defined by $q$ if and only if 
    \[d(q)=\begin{cases}
        -1 &\text{if $n\equiv 0 \bmod 4$}\\
        -3 & \text{if $n\equiv 2 \bmod 4$}
    \end{cases}\]
    and $\epsilon_p(q)=1$ for all primes $p$ such that $p\equiv 1 \bmod 3$.
\end{theorem}
\begin{proof}
    Let $H\cong (C_3)^k$ for $k\geq 1$. Every non-trivial irreducible rational representation $\sigma$ for $H$ is 2-dimensional and its image contains a matrix of order three. Thus Lemma~\ref{lem:dim2_reps} applies to show that every quadratic form in $\calQ(\sigma)$ is equivalent to a diagonal form $\langle 3a, a \rangle$ for some $a$.

    Since $b_1(B)=0$, we see that the holonomy representation can be decomposed as
    \[
    \rho\cong \rho_1 \oplus \dots \oplus \rho_{k}
    \]
    where each $\rho_i$ is 2-dimensional and irreducible over $\Q$ and $n=2k$. Thus $B$ appears in the commensurability class defined by $q$ if and only if $q$ is equivalent to
   \begin{equation}\label{eq:b1=0_3-torsion_form}
    q\cong \langle 3a_1, a_1, \dots, 3a_k, a_k, 1,-1\rangle.
    \end{equation}
    for some choice of $a_1,\dots, a_k\in\Q_{>0}$. Alternatively we may write this in the form 
    \[q\cong  3f\oplus f\oplus \langle 1,-1\rangle,\] where $f=\langle a_1, \dots, a_k\rangle$. For such a $q$ we have
    \[
    d(q)=\begin{cases}
        -1 & n\equiv 0 \bmod 4\\
        -3 & n\equiv 2 \bmod 4
    \end{cases}
    \]
    and
    \begin{align*}
        \epsilon_p(q')&=\epsilon_p(3f)\epsilon_p(f) (d(f),-1)_p(d(3f),-d(f))_p\\
        &=\begin{cases}
        (3,(-1)^{\frac{n}{4}})_p(d(f),-3)_p & n\equiv 0 \bmod 4\\
        (3,(-1)^{\frac{n-2}{4}})_p(d(f),-3)_p & n\equiv 2 \bmod 4.
    \end{cases}
    \end{align*}
    Since $(3,-1)_p=-1$ only for $p=2,3$ and $-3$ is not a square in $\Q_2$ or $\Q_3$, we see that $q'$ satisfies $\epsilon_p(q')=1$ for all $p$ such that $-3$ is a square in $\Q_p$. Conversely, since $f$ can be chosen to make $d(f)$ any positive rational, we see that any $q$ satisfying $d(q)=\begin{cases}
        -1 & n\equiv 0 \bmod 4\\
        -3 & n\equiv 2 \bmod 4
    \end{cases}$ and $\epsilon_p(q)=1$ for all $p$ such that $-3$ is a square in $\Q_p$ is equivalent to some $q'$ of the form given in \eqref{eq:b1=0_3-torsion_form}.
    \end{proof}
The case $n\equiv 2\bmod 4$ in Theorem~\ref{thm:Z3b1=0} is particularly interesting.
In this case the rank of the form $q$ is divisible by 4 and so its projective equivalence class is determined by the discriminant $d(q)$ and the Hasse-Witt invariants $\epsilon_p(q)$ at primes such that $d(q)$ is a square in $\Q_p$. However $-3$ is a square in $\Q_p$ if and only if $p\equiv 1 \bmod 3$, so Theorem~\ref{thm:Z3b1=0} determines all the projective invariants of $q$. In particular, there is a unique projective class satisfying these conditions. Examples of manifolds with holonomy $(C_3)^2$ and $b_1(B)=0$ exist in all even dimensions $n\geq 8$. An example of such a manifold in dimension $8$ is described in \cite[\S2]{Hiss2021quaternionic} (see also \cite{Hiller1986flat}). Examples in all higher even dimensions can then be constructed using toral extension (Proposition~\ref{prop:toral_extension}).

\begin{repcor}{cor:unique_class}
    \coruniqueclass \qed
\end{repcor}

\subsection{Holonomy of the form $(C_p)^k$}
Now we investigate manifolds with holonomy group of the form $(C_p)^k$ and $b_1(B)=0$. For every odd prime $p$ there is a compact flat manifold $B$ with $b_1(B)=0$, holonomy group $H=(C_p)^2$ and dimension $p^2-1$ \cite[Proposition~3.3]{Hiller1986flat}. Using toral extension, we are thus able to obtain a compact flat manifold with $b_1(B)=0$, holonomy group $H=(C_p)^2$ and dimension $p^2-1+k(p-1)$ for any $k\geq 0$.

Note that the following result is consistent with the calculation for $C_3$ carried out in Lemma~\ref{lem:dim2_reps}.
\begin{prop}\label{prop:prime_reps}
    Let $p$ be an odd prime. Let $\rho:C_p \rightarrow GL_{p-1}(\Q)$ be the unique non-trivial irreducible representation for $C_p$. Then any non-degenerate quadratic form $q\in \calQ(\rho)$ satisfies $d(q)=p \in \Q^\times/ (\Q^\times)^2$.
\end{prop}
\begin{proof}
    Let $\zeta$ denote a $p$th root of unity. The irreducible non-trivial representation of $C_p$ can be obtained considering the action of $\zeta$ on the field $\Q[\zeta]$ viewed as a vector space over $\Q$. Choosing the $\Q$-basis for $\Q[\zeta]$ given by $1,\zeta, \dots, \zeta^{p-2}$ means that $\zeta$ can be identified with the matrix
    
    \[
    \zeta=\begin{pmatrix}
        0 & 0 & \dots & 0 & -1 \\
        1 & 0 & \dots & 0 & -1 \\
        0 & 1 & \dots &0 & -1\\
        \vdots &\vdots &\ddots &\vdots&\vdots\\
        0 & 0& \dots& 1&-1
    \end{pmatrix}.\]
If we take a matrix $\alpha=\sum_i a_i \zeta^i$, then by definition $\det \alpha$ is equal to $N_{\Q[\zeta]/\Q}(\alpha)$, where $N_{\Q[\zeta]/\Q}$ denote the norm of the field extension $\Q[\zeta]/\Q$.

Since $\rho$ splits into $p-1$ distinct irreducible representations over $\C$, the endomorphism algebra $\End(\rho)$ is of dimension $p-1$. Thus, we see that $\End(\rho)=\Q[\zeta]$.

Let $Q\in\calQ(\rho)$ represent a non-zero quadratic form. By Proposition~\ref{prop:endo_algebra}, every other matrix representing a $\rho$-invariant bilinear form takes the form $Q\alpha$ for some $\alpha\in \End(\rho)$. However if $Q\alpha$ represents a quadratic form, then $\alpha$ must correspond to an element of the real field $\Q[\zeta]\cap \R$. Since $\zeta^T Q=Q\zeta^{-1}$, we see that $Q\alpha$ is symmetric only if
\[Q\alpha=\alpha^TQ=\sum_{k=0}^{p-2} a_k \zeta^kQ=\sum_{k=0}^{p-2} a_k Q\zeta^{-k}= Q\overline{\alpha}.\]
However, for any non-zero element $\alpha \in \Q[\zeta]\cap \R$, we can calculate the norm
\[
N_{\Q[\zeta]:\Q}(\alpha)=N_{\Q[\alpha]/\Q}(\alpha)^{[\Q[\alpha]:\Q[\zeta]]},
\]
to be a square in $\Q$ since the degree of the extension $\Q[\zeta]/\Q[\alpha]$, denoted $[\Q[\zeta]:\Q[\alpha]]$, is even. Thus we see that all non-zero elements of $\calQ(\rho)$ have the same discriminant. Thus it suffices to calculate the discriminant for a single form. The matrix
\[
Q= \begin{pmatrix}
    2 & -1    & 0 & 0 \\
    -1 & 2 & -1 & 0 \\ 
 0  &-1  & \ddots & -1  \\
 0  & 0 & -1    &2    
\end{pmatrix}.
\]
represents a form in $\calQ(\rho)$ since one can verify that $\zeta^T Q \zeta =Q$. One can calculate that we have $\det Q= p$.
\end{proof}

\begin{reptheorem}{thm:prime_analysis}
\thmprimeanalysis
\end{reptheorem}
\begin{proof}
    Let $H\cong (C_p)^k$ for some $k\geq 1$. Let $\sigma$ be a non-trivial irreducible rational representation for $H$. Such a representation is $(p-1)$-dimensional and restricts to a non-trivial representation on some $C_p$ subgroup. Thus Proposition~\ref{prop:prime_reps} shows that every non-degenerate quadratic form $q\in \calQ(\sigma)$ has discriminant $d(q)=p\in \Q^\times/(\Q^\times)^2$. Thus if $b_1(B)=0$, then the holonomy representation of $B$ has no trivial subrepresentations and can be decomposed as $\rho\cong \sigma_1 \oplus \dots \oplus \sigma_m$, where each $\sigma_i$ is $p-1$-dimensional and $n=m(p-1)$. If $B$ appears in the commensurability class defined by a form $q$, then $q$ is equivalent to
    \[
    q\cong f_1 \oplus \dots \oplus f_m \oplus \langle 1, -1 \rangle, 
    \]
    where each $f_i$ has discriminant $d(f_i)=p$. Thus we see that $q$ has discriminant
    \[
    d(q)=-p^m=\begin{cases}
        -1 &\text{if $n\equiv 0 \bmod 2(p-1)$}\\
        -p &\text{if $n\equiv p-1 \bmod 2(p-1)$}.
    \end{cases}
    \]
\end{proof}

In order to prove Corollary~\ref{cor:top_obstruction} we need to consider products of certain flat manifolds.

\begin{lemma}\label{lem:product_discriminant_property}
    Let $B_1$ and $B_2$ be compact flat manifolds and $d_1,d_2\in \Q_{<0}$ be rationals such that if $B_i$ is a cusp cross-section in a quasi-arithmetic hyperbolic manifold corresponding to a form $q_i$, then $d(q_i)=d_i\in \Q^\times/(\Q^\times)^2$. Then, if $B_1\times B_2$ is a cusp cross-section in a quasi-arithmetic hyperbolic manifold corresponding to a form $q$, then $d(q)=-d_1d_2\in \Q^\times/(\Q^\times)^2$.
\end{lemma}
\begin{proof}
Let $H_1$ and $H_2$ be the holonomy groups and $\rho_1$ and $\rho_2$ be rational holonomy representations for $B_1$ and $B_2$, respectively. The conditions on $B_1$ and $B_2$ imply that for all non-degenerate forms $f_i\in \calQ(\rho_i)$ we have $d(f_i)=-d_i$. We obtain a holonomy representation $\rho$ of $B_1 \times B_2$ by taking $\rho=(\rho_1\otimes \mathbf{1}_{H_2}) \oplus (\mathbf{1}_{H_1} \otimes \rho_2)$. Any element $f\in \calQ(\rho)$ takes the form $f=f_1 \oplus f_2$ for some $f_1\in \calQ(\rho_1)$ and $f_2\in \calQ(\rho_2)$. Consequently, we see that if $B_1\times B_2$ is a cusp cross-section in a quasi-arithmetic hyperbolic manifold corresponding to a form $q$, then $d(q)=-d_1d_2\in \Q^\times/(\Q^\times)^2$.
\end{proof}

\begin{repcor}{cor:top_obstruction}
    \cortopobstruction
\end{repcor}
\begin{proof}
    We take $B$ to be a 28-dimensional flat manifold with $b_1(B)=0$ and holonomy group $(C_5)^2$. For $k\geq 4$, we take $D_k$ to be an $2k$-dimensional flat manifold with $b_1(D_k)=0$ and holonomy group $(C_3)^2$. By Theorem~\ref{thm:prime_analysis}, if $D_{k}$ appears as a cusp cross-section in a quasi-arithmetic hyperbolic manifold, then the corresponding form must have discriminant $d(q)=-1$ or $-3$. By Theorem~\ref{thm:prime_analysis} and Lemma~\ref{lem:product_discriminant_property}, if $B\times D_k$ appears as a cusp cross-section in a quasi-arithmetic hyperbolic manifold, then the corresponding form must have discriminant $d(q)=-5$ or $-15$. Consequently, for $k\geq 2$ the manifolds $D_{18+k}$ and $B\times D_k$ are $36+2k$-dimensional flat manifolds that cannot occur as cusp cross-sections simultaneously in any quasi-arithmetic hyperbolic manifold.
\end{proof}
\bibliographystyle{alpha}
\bibliography{bib}
\end{document}